\documentclass[notitlepage,11pt,reqno]{amsart}

\usepackage{amsopn,esint,nicefrac}
\usepackage[final]{hyperref}
\usepackage[T1]{fontenc}
\usepackage{graphicx}

\usepackage[utf8]{inputenc}
\usepackage{apptools}
\AtAppendix{\counterwithin{lemma}{section}} 
\usepackage[margin=1in]{geometry}
\allowdisplaybreaks
\newcommand{\stkout}[1]{\ifmmode\text{\sout{\ensuremath{#1}}}\else\sout{#1}\fi}

\newcommand{\R}{\mathbb{R}}

\newcommand{\Rn}{\mathbb{R}^n}
\usepackage{graphicx,enumitem,dsfont,upgreek}
 \usepackage[dvips]{epsfig}
 \usepackage[mathscr]{eucal}
\usepackage{amscd}
\usepackage{amssymb,nicefrac}
\usepackage{amsthm}
\usepackage{amsmath}
\usepackage{latexsym}
\usepackage{dsfont}
\usepackage{upref}
\usepackage{hyperref}

\usepackage{color}
\theoremstyle{plain}

\newtheorem{thm}{Theorem}[section]
\theoremstyle{plain}
\newtheorem{lem}[thm]{Lemma}
\newtheorem{prop}[thm]{Proposition}

\newtheorem{defi}[thm]{Definition}
\newtheorem{rem}{Remark}[section]
\theoremstyle{definition}
\newtheorem*{maintheorem*}{Main Theorem}
\newtheorem*{maincorollary*}{Main Corollary}
{%
\setcounter{enumi}{0}

\begin{enumerate}}%
{\end{enumerate} }

{%
\setcounter{enumi}{0}

\begin{enumerate}}%
{\end{enumerate} }

\newcommand{\norm}[1]{\ensuremath{\left\|#1\right\|}}

\newcommand{\cone}{\ensuremath{\mathcal{C}}}
\newcommand{\cD}{\ensuremath{\mathcal{D}}}
\newcommand{\sL}{\ensuremath{\mathscr{L}}}

\newcommand{\cJ}{\ensuremath{\mathcal{J}}}

\newcommand{\dist}{{\rm dist}}
\newcommand{\xbar}{\ensuremath{\bar{x}}}
\newcommand{\ybar}{\ensuremath{\bar{y}}}
\newcommand{\abar}{\ensuremath{\bar{a}}}
\newcommand{\tail}{{\rm Tail}}
\newcommand{\data}{\texttt{data}}
\newcommand{\dataex}{\texttt{data}_e}

\newcommand{\dx}{\ensuremath{\, dx}}
\newcommand{\dy}{\ensuremath{\, dy}}
\newcommand{\dz}{\ensuremath{\, dz}}

\newcommand{\dr}{\ensuremath{\, dr}}

\numberwithin{equation}{section} \allowdisplaybreaks

\usepackage{cancel,pdfsync}

\title[Lipschitz regularity of fractional $p$-Laplacian]{Lipschitz regularity of fractional $p$-Laplacian}

\begin{document}

\author{Anup Biswas}

\address{Indian Institute of Science Education and Research-Pune, Dr.\ Homi Bhabha Road, Pashan, Pune 411008. INDIA Email:
{\tt anup@iiserpune.ac.in}}

\author{Erwin Topp}

\address{
Instituto de Matem\'aticas, Universidade Federal do Rio de Janeiro, Rio de Janeiro - RJ, 21941-909, BRAZIL.
Email: {\tt etopp@im.ufrj.br}
}

\begin{abstract}
In this article, we investigate the H\"{o}lder regularity of the fractional $p$-Laplace equation of the form $(-\Delta_p)^s u=f$ where $p>1, s\in (0, 1)$
and $f\in L^\infty_{\rm loc}(\Omega)$. Specifically, we prove that $u\in C^{0, \gamma_\circ}_{\rm loc}(\Omega)$ for 
$\gamma_\circ=\min\{1, \frac{sp}{p-1}\}$, provided that $\frac{sp}{p-1}\neq 1$. In particular, it shows that $u$ is locally Lipschitz for $\frac{sp}{p-1}>1$.
Moreover, we show that for $\frac{sp}{p-1}=1$, the solution is locally Lipschitz, provided that $f$ is locally H\"{o}lder continuous.
Additionally, we discuss further regularity results for the fractional double-phase problems.
\end{abstract}


\maketitle


\section{Introduction}
In this article we study the higher H\"{o}lder regularity of the local weak solution of fractional $p$-Laplacian equation given by 
\begin{equation}\label{I01}
(-\Delta_p)^su=f \quad \text{in}\; \Omega\subset\Rn,
\end{equation}
where $p>1, s\in (0, 1), n\geq 1$, and $f\in L^\infty_{\rm loc}(\Omega)$. Here the operator  $(-\Delta_p)^s u$ is formally defined by
\begin{equation*}
(-\Delta_p)^su(x) =  -(1-s)\, {\rm PV}\int_{\Rn} |u(y)-u(x)|^{p-2}(u(y)-u(x))\frac{dy}{|x-y|^{n+sp}},
\end{equation*}
where ${\rm PV}$ stands for the Cauchy Principal Value. 

This operator naturally arises
in the calculus of variations as the Euler-Lagrange equation of the inhomogeneous
 $W^{s,p}$ energy given by
\begin{equation}\label{I02}
u\mapsto \frac{1-s}{p}\iint_{(\Omega^c\times\Omega^c)^c} |u(y)-u(x)|^{p}\frac{dy\dx}{|x-y|^{n+ps}}-\int_{\Omega} fu \dx.
\end{equation}

The factor $(1-s)$ in the above expression ensures that the energy functional in  \eqref{I02} converges (in fact, $\Gamma$-converges) to the 
classical $p$-Dirichlet energy, given by
$$u\mapsto c_{n,p} \int_{\Omega} |\nabla u|^p \dx - \int_{\Omega} f u \dx,$$
as $s\nearrow 1$, where $c_{n,p}$ depends only on $n$ and $p$ (see \cite{BBM01}). That is why \eqref{I01} is considered as the fractional analogue of the classical $p$-Laplace equation.

Recently, in their influential work, Brasco, Lindgren, and Schikorra \cite{BLS18} for $p\geq 2$, and Garain and Lindgren \cite{GL24} for $1<p<2$,
have demonstrated that a local weak solution of \eqref{I01} is {\it almost} $\Gamma:=\min\{1, \frac{1}{p-1}(sp-\frac{n}{q})\}$-H\"{o}lder continuous, provided that
$f\in L^q_{\rm loc}(\Omega)$, that is, for any $\gamma<\Gamma$ we have $u\in C^{0, \gamma}_{\rm loc}(\Omega)$. 
Moreover, counterexamples provided in \cite{BLS18,GL24} support the {\it optimality} of $\Gamma$ in the sense that for 
$q\in (\frac{n}{sp}, \infty), \frac{1}{p-1}(sp-\frac{n}{q})<1,$ and for every $\varepsilon>0$ there
exists $f\in L^q_{\rm loc}(\Rn)$ such that the corresponding solution to \eqref{I01}, with $\Omega=\Rn$, is not in $C^{0, \Gamma+\varepsilon}(B_1)$. 
Furthermore, by \cite[Example~1.6]{BLS18}, it is known that for $p=2, s=\frac{1}{2}$, there exists $f\in L^\infty(\Omega)$ for which the corresponding
weak solution $u$ is not Lipschitz continuous. It should be noted that, in this case, we have $\frac{sp}{p-1}=1$.
This leads to the question of whether the solutions to \eqref{I01} are in $C^{0, \Gamma}_{\rm loc}$ when
$\frac{1}{p-1}(sp-\frac{n}{q})\neq 1$. Our main result of this article provides an answer to this question.

\begin{thm}\label{Main}
Let $u\in W^{s, p}_{\rm loc}(\Omega)\cap L^{p-1}_{sp}(\Rn)$ be a local weak solution to \eqref{I01} and $f\in L^\infty_{\rm loc}(\Omega)$. 
Also, let $\frac{sp}{p-1}\neq 1$. Then
$u\in C^{0, \gamma_\circ}_{\rm loc}(\Omega)$, where $\gamma_\circ=\min\{1, \frac{sp}{p-1}\}$.
\end{thm}

As mentioned above, the solution need not be Lipschitz when $\frac{sp}{p-1}= 1$. However, we may recover Lipschitz regularity by
imposing additional regularity on $f$. More precisely, we suppose that 
\begin{equation}\label{F}
\sup_{t\in (0, 1]} \omega_\Omega(t) |\log(t)|^{n+1+2p(1-s)}<\infty,
\end{equation}
where
$$
 \omega_\Omega(t)=\sup\{|f(x)-f(y)|\; :\; |x-y|\leq 1, \; x, y\in \Omega\}.
$$
\begin{thm}\label{Main2}
Let $\frac{sp}{p-1}= 1$ and $f$ satisfy \eqref{F}. Then every local weak solution 
$u\in W^{s, p}_{\rm loc}(\Omega)\cap L^{p-1}_{sp}(\Rn)$  to \eqref{I01} is locally Lipschitz. Moreover, if $f\in L^\infty_{\rm loc}(\Omega)$, then
$u$ satisfies a {\it log-Lipschitz} estimate, that is, there exists $\uptheta>0$ such that for any $\Omega_1\Subset\Omega$ we have
$$|u(x)-u(y)|\leq C |x-y|(1+|\log|x-y||^\uptheta)\quad \text{for all}\;\; x, y\in \Omega_1,$$
for some constant $C$.
\end{thm}
For $p=2$, an analogous result was established in \cite{CGT22,Sil12}. We note that the Ishii-Lions type argument in \cite{CGT22} relies on a certain scaling technique, which requires the underlying domain to be
 $\Rn$. However, we do not have the freedom to impose this condition in our setting.

\subsection{A brief history and our novelty}
The regularity of weak solutions to non-local equations involving the fractional $p$-Laplacian has become a prominent area of research in recent years. Early studies in this field focused primarily on local boundedness and H\"{o}lder regularity for small exponents, drawing on the De Giorgi-Nash-Moser theory
 (see for instance, \cite{APT24,BP16,Coz17,DKP16,KKP16}). H\"{o}lder regularity for viscosity solutions was explored by  \cite{DFP19,Lin16}, extending the ideas presented in \cite{Sil06}. Fine zero-order regularity estimates were addressed in \cite{KMS15a,KMS15b},
while global and boundary regularity are discussed in \cite{IMS16,IMS20,IM24,KLL23,KKP16}.

A major breakthrough in higher fractional differentiability for $p\geq 2$ appeared in the studies by Brasco and Lindgren \cite{BL17}
 which  was followed by the work of Brasco, Lindgren and Schikorra \cite{BLS18} who demonstrated optimal higher H\"{o}lder regularity of solutions for $p\geq 2$.
Specifically, they showed that any local weak solution to \eqref{I01} is almost $\Gamma$-H\"{o}lder. 
This marked the beginning of a wave of studies focused on the regularity of such solutions, including 
\cite{BDLM25,BDLMS24a,BDLMS24b,DKLN,DN25,GL24}, and references therein. It is also worth mentioning the recent work by 
B\"{o}gelein et al. \cite{BDLMS24a} where the H\"{o}lder exponent is improved, for $p\geq 2$, to any value strictly less than $\{1, \frac{sp}{p-2}\}$ for the fractional $p$-harmonic functions.
In case of $1<p<2$, it is also shown in \cite{BDLMS24b} that any locally bounded fractional $p$-harmonic function is almost Lipschitz for all $s\in (0, 1)$. This was then further 
extended in \cite{BS25} where the authors prove that fractional $p$-harmonic function is locally Lipschitz if $sp>(p-2)_+$.

Most of the works mentioned above rely on the method of difference quotients, particularly in the context of fractional $(s, p)$-Laplace. This approach, specifically for the fractional $p$-Laplacian, was first introduced by Brasco and Lindgren \cite{BL17}. 
Since the H\"{o}lder regularity in the above work follows as a consequence of Morrey-type embedding, they fail to address the H\"{o}lder regularity at the exponent $\Gamma$. In this article, we take a completely
different approach, which is based on the viscosity solutions. This shift in methodology necessitates considering $f\in L^\infty_{\rm loc}(\Omega)$, leading to $\Gamma=\gamma_\circ=\min\{1, \frac{sp}{p-1}\}$. Thanks to the results in \cite{KKL19}, the equivalence between weak solutions and viscosity solutions ensures that our findings also apply to weak solutions. 
We employ an Ishii-Lions type argument, originally introduced by Ishii and Lions in \cite{IL90}, to establish H\"{o}lder regularity for viscosity solutions of nondegenerate elliptic second-order equations. In the same second-order framework, Capuzzo-Dolcetta, Leoni, and Porretta \cite{CDLP10} obtained interior Lipschitz estimates by revisiting Bernstein’s technique and integrating elements of the Ishii-Lions method (see also Barles \cite{Bar91}). Recently, this approach has been successfully extended to nonlocal operators to establish H\"{o}lder and Lipschitz regularity,
 as demonstrated in \cite{BCI11,BCCI12,BT24}. Notably, this method can also be adapted to investigate Liouville-type properties in the presence of gradient nonlinearity \cite{BQT25}. It is important to note that in all previous works, the nonlocal operator was either linear or of the Pucci type (that is, sub-additive), a structure that facilitated certain cancellations, see \cite[Lemma~23]{BCCI12}, \cite[Lemma~2.4]{BQT25}.
However, this structural advantage does not apply to the fractional $p$-Laplacian. To address this challenge, we approach the problem with great care, enabling us to successfully apply the Ishii-Lions argument (refer to the estimate of $I_3$). 
While it is possible to offer a slightly simpler proof of Lipschitz regularity using the established H\"{o}lder regularity results from \cite{BLS18,GL24}, we present self-contained proofs by employing an appropriate bootstrapping argument, assuming the solution to be continuous. Since the regularity results in \cite{Coz17,KMS15a} ensure continuity of local weak solutions, 
our technique is also applicable to weak solutions.
This choice enhances the applicability of our method to other models.
Another notable aspect of this method is its versatility: it can be extended to a broad range of similar operators. As discussed in Section ~\ref{S-other}, our approach applies to the fractional $(p,q)$-Laplacian and fractional double-phase problems. Notably, our estimate provides significantly sharper H\"{o}lder estimates compared to those previously available in the literature  (see Theorem \ref{T2.10} and \ref{T2.11}). 
We also believe that similar regularity for $f\in L^q_{\rm loc}(\Omega)$ can be established using the perturbation method, but we do not pursue this direction in this article.
Since the Ishii-Lions method provides a first-order approximation, our approach only yields Lipschitz regularity and may not be sufficient to prove 
$C^{1,\alpha}$ regularity.  However, considering the  $C^{1, \alpha}$ estimate for the classical $p$-Laplace operator (see \cite{U77,U68}), it
is natural to expect $C^{1,\alpha}$ regularity for the fractional $p$-harmonic function when $s>\frac{p-1}{p}$, a result that has remained unknown to date.

\subsection{General setting and notion of solution.}

Let $K:\Rn\setminus\{0\}\to (0, \infty)$ be a symmetric, continuous function satisfying
$$\frac{\lambda}{|z|^{n+ps}}\leq K(z)\leq \frac{\Lambda}{|z|^{n+ps}}\quad \text{for}\; z\neq 0,$$
where $0<\lambda\leq \Lambda$, $p>1$ and $s\in (0, 1)$. We are interested in
the operator
$$
\sL u(x):= {\rm PV}\int_{\Rn} |u(x)-u(y)|^{p-2}(u(x)-u(y)) K(x-y) \dy.
$$
The case $K(z)=|z|^{-n-ps}$ corresponds to the fractional $p$-Laplace operator.
Also note that we do not use the factor $(1-s)$ multiplied with the operator since our results are not stable as $s\nearrow 1$ (see Remark~\ref{Rem2.1}). Since we will be dealing with the viscosity solutions to
the above operator, it is necessary to define the appropriate notion of viscosity solution for $\sL$. To do so, we recall a few notation from 
\cite{KKL19}. By $L^{p-1}_{sp}(\Rn)$ we denote the \textit{tail space}  defined by
$$L^{p-1}_{sp}(\Rn)=\left\{f\in L^{p-1}_{\rm loc}(\Rn)\; :\; \int_{\Rn}\frac{|f(z)|^{p-1}}{(1+|z|)^{n+sp}}\dz <\infty\right\}.$$
Associated to this tail space we also define the {\it tail function} given by
$$
\tail(f; x, r) := \left(r^{sp} \int_{|z-x|\geq r}\frac{|f(z)|^{p-1}}{|z-x|^{n+sp}}\dz \right)^{\frac{1}{p-1}}\quad r>0.
$$
Given an open set $D$, we denote by $C^2_\beta(D)$, a subset of $C^2(D)$, defined as
$$
C^2_\beta(D)=\left\{\phi\in C^2(D)\; :\; \sup_{x\in D}\left[\frac{\min\{d_\phi(x), 1\}^{\beta-1}}{|\nabla\phi(x)|} +
\frac{|D^2\phi(x)|}{(d_\phi(x))^{\beta-2}}\right]<\infty\right\},
$$
where
$$ 
d_\phi(x)=\dist(x, N_\phi)\quad \text{and}\quad N_\phi=\{x\in D\; :\; \nabla\phi(x)=0\}.$$

The above restricted class of test functions becomes necessary to define $\sL$ in the classical sense in the singular case, that is, for
$p\leq \frac{2}{2-s}$. Now we are ready to define the viscosity solution from \cite[Definition~3]{KKL19}.

\begin{defi}\label{Def1.1}
A function $u:\Rn\to \R$ is a viscosity subsolution to $\sL u =f $ in $\Omega$ if it satisfies the following
\begin{itemize}
\item[(i)] $u$ is upper semicontinuous in $\bar\Omega$.
\item[(ii)] If $\phi\in C^2(B_r(x_0))$ for some $B_r(x_0)\subset \Omega$ satisfies $\phi(x_0)=u(x_0)$,
$\phi\geq u$ in $B_r(x_0)$ and one of the following holds
\begin{itemize}
\item[(a)] $p>\frac{2}{2-s}$ or $\nabla\phi(x_0)\neq 0$,

\item[(b)] $1<p\leq \frac{2}{2-s}$, $\nabla\phi(x_0)= 0$ is such that $x_0$ is an isolated critical point of $\phi$, and
$\phi\in C^2_\beta(B_r(x_0))$ for some $\beta>\frac{sp}{p-1}$,
\end{itemize}
then we have $\sL \phi_r(x_0)\leq f(x_0)$ where
\[
\phi_r(x)=\left\{\begin{array}{ll}
\phi(x) & \text{for}\; x\in B_r(x_0),
\\[2mm]
u(x) & \text{otherwise}.
\end{array}
\right.
\]

\item[(iii)] $u_+\in L^{p-1}_{sp}(\Rn)$.
\end{itemize}
We say $u$ is a viscosity supersolution in $\Omega$, if $-u$ is a viscosity subsolution in $\Omega$. Furthermore, a 
viscosity solution to $\sL u= f$ in $\Omega$ is both sub and super solution in $\Omega$. 
\end{defi}

It is clear from the definition that it is equivalent to general testing, in the sense that $x_0$ is a local maximum point for $u - \phi$ (and not necessarily $u(x_0) = \phi(x_0)$). Same comment goes for the testing of supersolutions.

\medskip

We also need the notion of weak solution for $\sL$, and for this, we need the space
$$
W^{s, p}(\Omega)=\left\{ v\in L^p(\Omega)\; :\; \int_\Omega\int_\Omega \frac{|v(x)-v(y)|^p}{|x-y|^{n+sp}}\dx\dy<\infty \right\},
$$
and the corresponding local fractional Sobolev space, defined by
$$
W^{s, p}_{\rm loc}(\Omega)=\{ v\in L^p_{\rm loc}(\Omega)\; :\;  v\in W^{s, p}(\tilde\Omega)\; \text{for any}\; \tilde\Omega\Subset\Omega \}.
$$
We also denote 
$$
J_p(t)= |t|^{p-2} t, \quad t\in\R.
$$
\begin{defi}
A function $u\in W^{s, p}_{\rm loc}(\Omega)\cap L^{p-1}_{sp}(\Rn)$ is said to be subsolution to $\sL u =f$ in $\Omega$ if we have
$$ \frac{1}{2}\int_{\Rn}\int_{\Rn} J_p(u(x)-u(y))(\phi(x)-\phi(y))K(x-y)\dx\dy\leq \int_{\Rn} f(x) \phi(x) \dx,$$
for all non-negative $\phi\in C^\infty_0(\Omega)$.
\end{defi}
Analogously, we can also define weak super-solution and solution. Equivalence between the weak supersolution and the viscosity supersolution are established in
\cite{KKL19}. Here we add a proof for convenience. The continuity of the kernel $K$ is used for this result. 

\begin{prop}\label{Prop1.3}
Let $\Omega$ be an open set and $f\in C(\Omega)$. Let $u \in W^{s, p}_{\rm loc}(\Omega)\cap L^{p-1}_{sp}(\Rn)$ be a weak 
subsolution to $\sL u=f$ in $\Omega$. Furthermore, assume that $u$ is upper-semicontinuous in $\bar\Omega$. Then $u$ is a viscosity subsolution
to $\sL u=f$ in $\Omega$. An analogous conclusion holds for weak supersolution and solution.
\end{prop}

\begin{proof}
Note that $u$ satisfies conditions (i) and (iii) of Definition~\ref{Def1.1}. Consider a test function $\phi\in C^2(B_r(x_0))$ with $B_r(x_0)\subset \Omega$
such that $\phi(x_0)=u(x_0)$, $\phi\geq u$ in $B_r(x_0)$ and either (a) or (b)  in Definition~\ref{Def1.1}(iii) holds. We need to show that
$\sL\phi_r(x_0)\leq f(x_0)$, where
\[
\phi_r(x)=\left\{\begin{array}{ll}
\phi(x) & \text{for}\; x\in B_r(x_0),
\\[2mm]
u(x) & \text{otherwise}.
\end{array}
\right.
\]
Suppose, to the contrary, that $\sL\phi_r(x_0)> f(x_0)$. Using the continuity of $f$ and $\sL\phi$ in $B_r(x_0)$
(see \cite[Lemma~3.8]{KKL19}), we can find $r_1\in (0, r)$ and
$\epsilon>0$ such that $\sL\phi_r(x)> f(x) + \epsilon$ in $B_{r_1}(x_0)$. By \cite[Lemma~3.9]{KKL19} there exist $r_2\in (0, r_1), \hat\delta>0$, and
a non-negative function $\chi\in C_0^2(B_{r_2/2}(x_0))$ satisfying $\chi(x_0)=1$, $0\leq \chi\leq 1$, such that for $\phi_\delta=\phi_r-\delta\chi$ we have
$$ \sup_{B_{r_2}(x_0)}|\sL \phi_r - \sL \phi_\delta|<\frac{\epsilon}{2}$$
for $\delta\in (0, \hat\delta)$. Furthermore, for $1<p\leq \frac{2}{2-s}$, we will have $\phi_\delta\in C^2_\beta(B_{r_2}(x_0))$ for $\beta>\frac{sp}{p-1}$. Therefore,
$\sL\phi_\delta> f+\epsilon/2$ in $B_{r_2}(x_0)$.
 The proof
of \cite[Lemma~3.10]{KKL19} reveals that $\phi_\delta$ is a weak supersolution to $\sL \phi_\delta=f$ in $B_{r_2}(x_0)$. Since $\phi_\delta\geq u$ in $B^c_{r_2}(x_0)$, from the
comparison principle \cite[Lemma~6]{KKP17} (see also, \cite[Lemma 9]{LL14}), we obtain $\phi_\delta\geq u$ in $\Rn$. But this is not possible since $\phi_\delta(x_0)=\phi(x_0)-\delta<u(x_0)$.
Thus, we must have $\sL\phi_r(x_0)\leq f(x_0)$. Hence the proof.
\end{proof}

The rest of the article is organized as follows. 
In the next section, we introduce  our first main result (Theorem~\ref{Tmain-1}) within the viscosity framework and outline the general strategy behind our proofs.
Section~\ref{secsuper} presents the proof of Theorem~\ref{Tmain-1} in the superquadratic case, while the subquadratic case is addressed in Section~\ref{S-sub}.
In Section~\ref{S-other}, we explore several extensions of our result, including fractional $(p,q)$-Laplacian and
fractional double-phase problems. Finally, in  Section~\ref{S-crit}, we provide the proof of the Lipschitz regularity for the critical case (Theorem~\ref{Main2}).

In what follows, $\kappa, \kappa_1, \kappa_2, \ldots$ denote arbitrary constants that may vary from line to line.

\section{Ishii-Lions methods for fractional $p$-Laplacian.}

We now state the main result Theorem~\ref{Main} in a more precise way, in the viscosity framework.
\begin{thm}\label{Tmain-1}
Let $p\in (1, \infty)$, $s\in (0,1)$ and define $\gamma_\circ(s)=\min\{1, \frac{sp}{p-1}\}$. Let $u\in C(\bar{B}_2)\cap L^{p-1}_{sp}(\Rn)$ be a viscosity solution to
$$ -C \leq \sL u\leq C\quad \text{in}\; B_2.$$
Then $u\in C^{0, \gamma}(B_1)$ for every $\gamma<\gamma_\circ(s)$, and for $\gamma<\gamma_\circ(s)$,
 we have
$$\norm{u}_{C^{0,\gamma}(B_1)}\leq C_1 \left(\sup_{B_2}|u| +  \tail(u; 0, 2) + C^{\frac{1}{p-1}} \right)$$
for some constant $C_1$, dependent only on $\gamma, \lambda, \Lambda, n, p$ and $s$. Furthermore, for $\frac{sp}{p-1}\neq 1$, we have 
$u\in C^{0, \gamma_\circ(s)}(B_1)$
and 
$$\norm{u}_{C^{0,\gamma_\circ(s)}(B_1)}\leq C_2 \left(\sup_{B_2}|u| +  \tail(u; 0, 2) + C^{\frac{1}{p-1}} \right)$$
for some constant $C_2$, dependent only on $\gamma, \lambda, \Lambda, n, s$ and $p$. 
\end{thm}

\begin{rem}\label{Rem2.1}
$(1)$ The constants $C_1$ and $C_2$ appearing in Theorem~\ref{Tmain-1} are not stable if we let $s\nearrow 1$, even if we multiply the operator $\sL$ with an additional normalizing factor
$(1-s)$. The problem arises due to the term $\varepsilon^{p(1-s)}_1$ in the lower bound of $I_2$ in Lemma~\ref{lemD1}. Since $s\in (0, 1)$, it helps us  to choose 
$\varepsilon_1$ small enough so that $I_1$ dominates $I_2$ (see \eqref{E1.4}).

\medskip

\noindent
$(2)$ If $ -C \leq \sL u\leq C$ in $B_{2R}$ for some $u\in C(\bar{B}_{2R})\cap L^{p-1}_{sp}(\Rn)$, then letting
$v(x)=u(Rx)$, it is easily seen that 
$$ - R^{sp} C\leq \tilde\sL v\leq R^{sp} C \quad \text{in}\; B_2,$$
in the viscosity sense, where $\tilde\sL$ is defined with respect to the kernel $\tilde{K}(z)=R^{-n-sp}K(Rz)$. By Theorem~\ref{Tmain-1}
we then have
$$\norm{u}_{C^{0,\gamma}(B_R)}\leq \frac{C_1}{R^\gamma} \left(\sup_{B_{2R}}|u| +  \tail(u; 0, 2R) + (R^{sp}C)^{\frac{1}{p-1}} \right).$$
\end{rem}

Proof of Theorem~\ref{Main} follows from Proposition~\ref{Prop1.3} and Theorem~\ref{Tmain-1}.
\begin{proof}[Proof of Theorem~\ref{Main}]
First we observe that $u\in C(\Omega)$. 
For $sp>n$, this follows from the Sobolev-Morrey  embedding \cite{DNPV}, whereas for $sp\leq n$, it follows from the regularity result in
\cite[Theorem~2.4]{Coz17}(see also, \cite[Corollary~1.2]{KMS15a}, \cite[Theorem~3.13]{BP16}).
Hence the proof follows by combining Proposition~\ref{Prop1.3}, Theorem~\ref{Tmain-1} together with a scaling (see Remark~\ref{Rem2.1}) and a standard covering argument.
\end{proof}

In what follows, we concentrate on the proof Theorem~\ref{Tmain-1}.

\subsection{General Strategy.}\label{subsecGS}
In what follows, we explain the general strategy based on Ishii-Lions method in order to conclude Theorem~\ref{Tmain-1}.

Let $M=\sup_{B_2}|u| +  \tail(u; 0, 2) + C^{\frac{1}{p-1}}$ and without loss of any generality we may assume that $M>0$.
Replacing $u$ by $u/M$, it is enough to consider the following situation
\begin{equation}\label{E1.1}
\left \{ \begin{array}{l}
-1\leq \sL u\leq 1 \quad \text{in}\; B_2,
\\[2mm]
\sup_{B_2}|u| +  \tail(u; 0, 2) \leq 1. 
\end{array} \right .
\end{equation}

Fix $1\leq \varrho_1<\varrho_2\leq 2$, and define the doubling
function
\begin{equation}\label{E1.2}
\Phi(x, y)= u(x)-u(y)-L\varphi(|x-y|)- m_1 \psi(x)\quad x, y\in B_2,
\end{equation}
where 
$$
\psi(x) = [(|x|^2-\varrho^2_1)_+]^{m}, \ x \in B_2
$$ 
is a {\it localization} function. We set $m\geq 3$ so that $\psi\in C^2(B_2)$. Though not explicitly depending on $\varrho_2$, this parameter is going to be involved with $\Phi$ through $m_1$ as we will see next.

For $r_0 \in (0,1)$ to be precisely defined in each case, the function $\varphi:[0, r_0) \to [0, \infty)$ is a suitable {\it regularizing} function, encoding the modulus of continuity for $u$. 

For calculations below, we use two types of regularizing function $\varphi$:
\begin{equation}\label{thevarphis}
\begin{split}
& \varphi_\gamma(t)=t^{\gamma}\quad \mbox{with} \ \gamma\in (0, 1) \quad \mbox{(H\"older profile for case $sp < p-1$)}, 
\\
& \tilde\varphi(t)=t+\frac{t}{\log t} \quad  \mbox{(Lipschitz profile for case $sp \geq p-1$)},
\end{split}
\end{equation}
where the last function is defined at zero as $\tilde \varphi(0) = 0$. Notice that the above functions are increasing and concave on a neighborhood of $t = 0$. 

In general, we show that for a sufficiently large $m_1$,
and for all $L$ large enough but independent of $u$, we have $\Phi\leq 0$ in $B_2\times B_2$, which leads to the desired result.

We proceed by contradiction, assuming that $\sup_{B_2\times B_2}\Phi>0$ for all large $L$. Let us choose $m_1$ large enough, dependent on
$\varrho_2, \varrho_1$ and $m$, so that $m_1 \psi(x)>2$ for all $|x|\geq \frac{\varrho_2+\varrho_1}{2}$. Then for all 
$|x|\geq \frac{\varrho_2+\varrho_1}{2}$, we have $\Phi(x, y)<0$ for all $y\in B_2$. Again, since $\varphi$ is strictly increasing in
$[0, 2]$, if we choose $L$ to satisfy $L\varphi(\frac{\varrho_2-\varrho_1}{8})>2$, we obtain $\Phi(x, y)<0$ whenever
$|x-y|\geq \frac{\varrho_2-\varrho_1}{8}$. Thus, there exists $\xbar\in B_{\frac{\varrho_2+\varrho_1}{2}}$ and 
$\ybar\in B_{\frac{5\varrho_2}{8}+\frac{3\varrho_1}{8}}$ such that
\begin{equation}\label{E1.3}
\sup_{B_2\times B_2}\Phi=\Phi(\xbar,\ybar)>0.
\end{equation}
Denote by $\abar=\xbar-\ybar$. From \eqref{E1.3} we have $\abar\neq 0$, and moreover, we have that
$$
L \varphi(|\bar a|) \leq u(\bar x) - u(\bar y) \leq 2,
$$
in view of~\eqref{E1.1}. This implies that $|\bar a|$ gets smaller as $L$ enlarges.
Also, denote 
$$
\phi(x, y) := L\varphi(|x-y|)+ m_1 \psi(x).
$$
Note that
\begin{center}
$x\mapsto u(x) - \phi(x, \ybar)$ has a local maximum point at $\xbar$, and\\[2mm]
$y \mapsto u(y) +\phi(\xbar,y)$ has a local minimum point at $\ybar$.
\end{center}

For $\delta\in (0, \frac{\varrho_2-\varrho_1}{4})$ to be chosen later, we define the following test functions
\[
w_1(z)=\left\{\begin{array}{ll}
\phi(z, \ybar) & \text{if}\; z\in B_\delta(\bar x),
\\[2mm]
u(z) & \text{otherwise},
\end{array}
\right.
\]
 and
\[
w_2(z)=\left\{\begin{array}{ll}
-\phi(\xbar, z) & \text{if}\; z\in B_\delta(\bar y),
\\[2mm]
u(z) & \text{otherwise}.
\end{array}
\right.
\]

The choice of $\delta$, which would depend on $\abar$, will of course ensure that $w_1, w_2$ are $C^2$ around $\bar x, \bar y$ respectively .
An important point here is that, regardless of the choice of $\varphi$ in~\eqref{thevarphis}, for all sufficiently large
$L$ (depending on $m$ and $m_1$), we must have $\nabla_x\phi(\xbar,\ybar)\neq0$ and $\nabla_y\phi(\xbar,\ybar)\neq 0$. Thus, from Definition~\ref{Def1.1}
and \eqref{E1.1}, we obtain
$$
\sL w_1(\bar x)\leq 1\quad \text{and}\quad \sL w_2(\bar y)\geq -1.
$$

Furthermore, as can be seen from \cite{KKL19}, the above principal values are well defined.
We now introduce the notation
\begin{equation}\label{AB01}
\sL[D] w(x)=  {\rm PV}\int_D |w(x)-w(x+z)|^{p-2}(w(x)-w(x+z))K(z)\dz,
\end{equation}
where the Principal Value agrees with the actual integration if $\dist(0, D)>0$. 

Now, define the following domains
$$\cone=\{z\in B_{\delta_0|\abar|}\; :\; |\langle \abar, z\rangle|\geq (1-\delta_0) |\abar||z|\},
\quad \cD_1=B_\delta \cap \cone^c, \quad \text{and}\quad \cD_2=B_{\tilde\varrho}\setminus (\cD_1\cup\cone),$$
where $\delta_0\in (0, \frac{1}{2})$ would be chosen later, $\tilde\varrho=\frac{1}{4}(\varrho_2-\varrho_1)$, and, in general, $\delta << \delta_0 |\bar a| << \tilde \varrho$. See Figure~\ref{figure} .
\begin{figure}[ht]
\begin{center}
\includegraphics[scale=0.3]{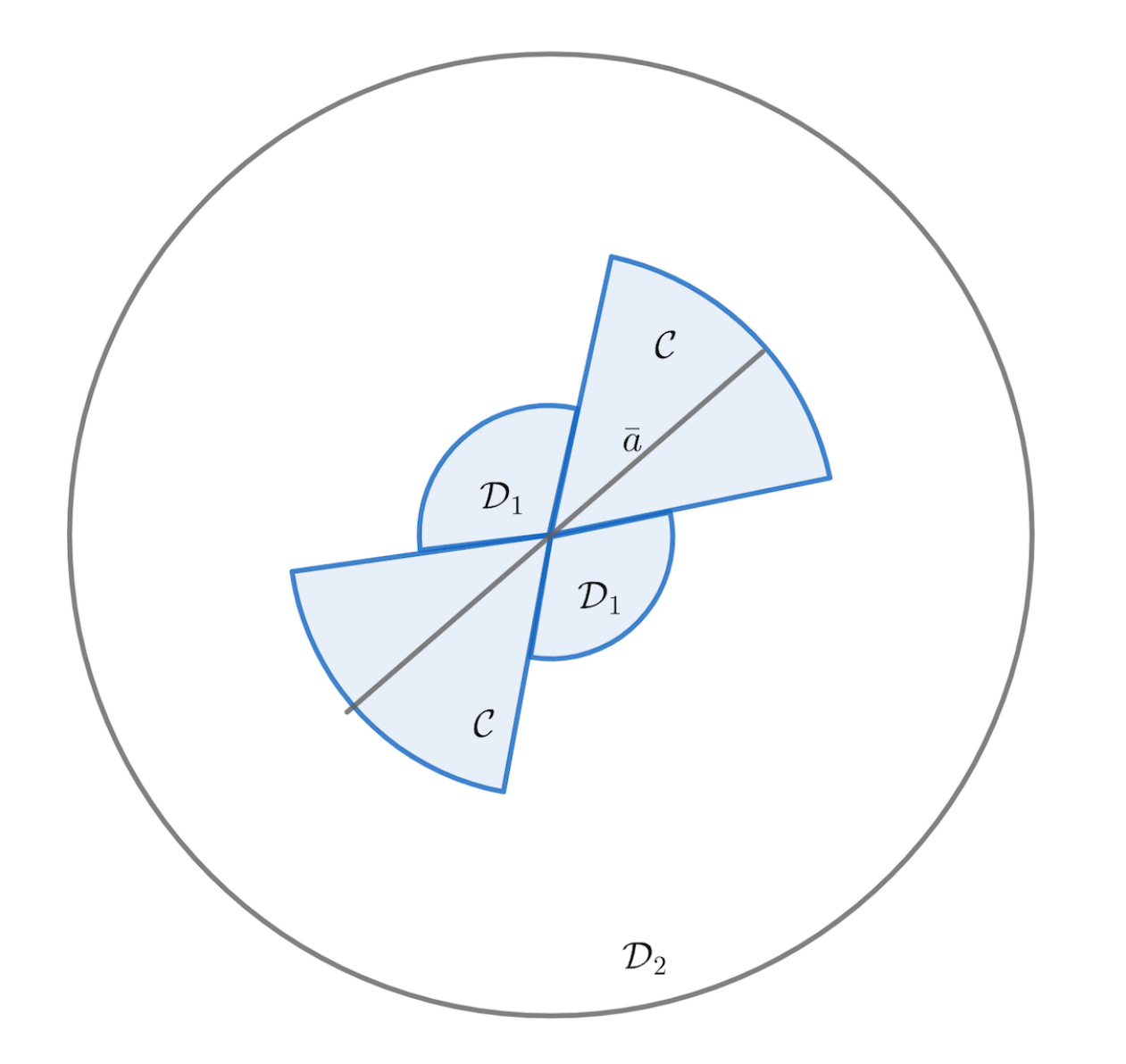}
\end{center}
\caption{Relevant sets for the nonlocal viscosity evaluation.}
\label{figure}
\end{figure} 
Subtracting the viscosity inequalities at $\bar x$ and $\bar y$, we obtain
$$\sL w_1(\bar x)-\sL w_2(\bar y)\leq 2,$$
and using the above notation we write
\begin{align}\label{E1.4}
&\underbrace{\sL[\cone] w_1(\bar x)-\sL[\cone] w_2(\bar y)}_{=I_1} + \underbrace{\sL[\cD_1] w_1(\bar x)-\sL[\cD_1] w_2(\bar y)}_{=I_2} + \nonumber
\\ 
&\quad \underbrace{\sL[\cD_2] w_1(\bar x)-\sL[\cD_2] w_2(\bar y)}_{=I_3}
+ \underbrace{\sL[B^c_{\tilde\varrho}] w_1(\bar x)-\sL[ B^c_{\tilde\varrho}] w_2(\bar y)}_{=I_4}\leq 2.
\end{align}

Our chief goal is to compute $I_i, i=1,2, 3, 4$ suitably and derive a contradiction from \eqref{E1.4} for large enough $L$. For this purpose, we present some technical estimates which will lead us to the desired conclusion.

\subsection{Basic technical estimates.} We start with the following estimates on the cone $\cone$.
\begin{lem}\label{L1.5}
For $0<|\abar|\leq \frac{1}{8}$, consider the cone 
$\cone=\{z\in B_{\delta_0|\abar|}\; :\; |\langle \abar, z\rangle| \geq (1-\delta_0) |\abar||z|\}$, where $\delta_0\in (0, 1)$.
Set 
\begin{equation}\label{defphi}
\phi(x, y)=L\varphi(|x-y|)+ m_1\psi(x),
\end{equation}
and for a differentiable function $f: \R^n \to \R$, $x, z \in \R^n$, denote
\begin{equation}\label{theta2}
\Theta_2 f(x, z) = f(x) - f(x + z) + \nabla f(x) \cdot z.
\end{equation}

Then
\begin{itemize}
\item[(i)] For $\varphi(t)=\varphi_\gamma(t)=t^\gamma$, $\gamma\in (0, 1)$, there exist $L_0, \varepsilon > 0$, independent of $\abar$, such that for all $\delta_0\in (0, \varepsilon]$ and all $L\geq L_0$ we have
$$ \frac{1}{\kappa} L |\abar|^{\gamma-2} |z|^2 \leq \Theta_2 \phi(\cdot,\ybar)(\bar x, z)
\leq \kappa L |\abar|^{\gamma-2} |z|^2,\qquad z\in\cone
$$
for some constant $\kappa$ depending on $\gamma$. 

\item[(ii)] Denote 
\begin{align*}
	\tilde \varphi(t) = \left \{ \begin{array}{cl} t + \frac{t}{\log t} \quad & \mbox{if} \ t \in (0, r_o), \\ 0 \quad & \mbox{if} \ t = 0, \end{array} \right .
	\end{align*}
	for some $r_\circ \in (0,1)$ small enough. 
    For $\varphi(t) = \tilde \varphi(\frac{r_\circ}{3}t)$,  there exist $L_0, \varepsilon > 0$, independent of $\abar$,
 such that for $\delta_0=\delta_1(\log^2|\abar|)^{-1}$ and 
 $\delta_1\in (0, \varepsilon]$, $L \geq L_0$,
we have
	\begin{equation*}
		\frac{1}{\kappa} L (|\abar|\log^2(|\abar|))^{-1} |z|^2 \leq \Theta_2 \phi(\cdot, \bar y)(\bar x, z)
		\leq \kappa L (|\abar|\log^2(|\abar|))^{-1} |z|^2,\qquad z\in\cone,
	\end{equation*}
where $\kappa$ depends on $r_\circ$.
	
\item[(iii)] Same estimates hold for $\Theta_2 \phi(\bar x, \cdot)(\bar y, z)$ in all the cases above.

\end{itemize}
\end{lem}

\begin{proof}
We start with $(i)$. Using Taylor's expansion, we note that
\begin{equation}\label{EL1.5A}
 -\frac{1}{2}\sup_{|t|\leq 1}\langle z,  D^2_x \phi(\xbar+tz,\ybar) z\rangle \leq \phi(\xbar,\ybar)-\phi(\xbar+z,\ybar) + \nabla_x\phi(\xbar,\ybar)\cdot z\leq -\frac{1}{2}\inf_{|t|\leq 1} \langle z, D^2_x \phi(\xbar+tz,\ybar) z\rangle.
\end{equation}

A straightforward calculation reveals
\begin{align*}
\langle z, D^2_x \phi(\xbar+tz,\ybar) z\rangle &= L\left\{ \frac{\varphi'(|\abar + tz|)}{|\abar+tz|}(|z|^2-\langle\widehat{\abar+tz}, z\rangle^2) + \varphi^{\prime\prime}(|\abar+tz|)\langle\widehat{\abar+tz}, z\rangle^2 \right\}
\\
&\quad + m_1 \langle z, D^2_x \psi(\xbar+tz,\ybar) z\rangle.
\end{align*}

Here and in what follows, $\hat{e}$ denotes the unit vector in the direction of $e\neq 0$.
For $z\in\cone$ and $|t|\leq 1$ we see that
$$| \langle\abar+tz, z\rangle|\geq (1-2\delta_0)|\abar||z|,$$
and
$$
(1-\delta_0) |\abar|\leq |\abar+tz|\leq (1+\delta_0)|\abar|.
$$
We choose $\varepsilon$ small enough so that for $\delta_0\in (0, \varepsilon]$ we have
$$
\eta:=\frac{1-2\delta_0}{1+\delta_0}\geq \frac{1}{\sqrt{2}}\quad \text{and}\quad \frac{1-\eta^2}{1-\delta_0}\leq 12\delta_0.
$$

With this choice, for $z\in\cone$,  we have
$$
\langle\widehat{\abar+tz}, z\rangle^2\geq \eta^2 |z|^2 \; \Rightarrow\quad
\frac{1}{|\abar + t z|} (|z|^2-\langle\widehat{\abar+tz}, z\rangle^2)\leq \frac{(1-\eta^2)|z|^2}{(1-\delta_0)|\abar|}\leq 12\delta_0 \frac{|z|^2}{|\abar|}.
$$

Note that this choice of $\varepsilon$ is independent of $L$ and $|\abar|$. Since $\varphi'_\gamma>0$ and $\varphi^{\prime\prime}_\gamma\leq 0$ in $(0, 2]$, we obtain
for $z\in\cone$ that
\begin{align}\label{EL1.5B}
&\left[ -L\frac{\gamma(1-\gamma)}{(1-\delta_0)^{2-\gamma}}|\abar|^{\gamma-2} -m_1 \sup_{B_2}|D^2\psi| \right] |z|^2\nonumber
\\
&\quad \leq \langle z, D^2_x \phi(\xbar+tz,\ybar) z\rangle \nonumber
\\
&\qquad\leq \left[12L\gamma\frac{\delta_0}{(1-\delta_0)^{1-\gamma}} |\abar|^{\gamma-2} -L\frac{\gamma(1-\gamma)}{2(1+\delta_0)^{2-\gamma}}
|\abar|^{\gamma-2} + m_1 \sup_{B_2}|D^2\psi| \right]|z|^2,
\end{align}
where for the lower bound we used $(|z|^2-\langle\widehat{\abar+tz}, z\rangle^2)\geq 0$. Now for fixed $\gamma\in (0, 1)$, we can choose $\varepsilon$ further small, if required, depending on $\gamma$, and
$L_0=L_0(\gamma, \varepsilon, m_1)$ so that
$$4 \gamma\frac{3\delta_0}{(1-\delta_0)^{1-\gamma}}- \frac{\gamma(1-\gamma)}{2(1+\delta_0)^{2-\gamma}} + \frac{|\abar|^{2-\gamma}}{L}m_1 \sup_{B_2}|D^2\psi|
< -\frac{\gamma(1-\gamma)}{4 2^{2-\gamma}},$$
and
$$-\frac{\gamma(1-\gamma)}{(1-\delta_0)^{2-\gamma}}- \frac{|\abar|^{2-\gamma}}{L} m_1 \sup_{B_2}|D^2\psi|\geq -2\gamma (1-\gamma),$$
for all $L\geq L_0$ and $\delta_0\leq \varepsilon$.
Combining \eqref{EL1.5A} and \eqref{EL1.5B} we have the estimate of $(i)$.

\medskip

\noindent
$(ii)$ \, Let $r_\circ>0$ be such that for $r\in (0, r_\circ]$ the following hold:
\begin{align*}
	\frac{r}{2} & \leq \tilde\varphi(r)\leq r,\quad
	\frac{1}{2}  \leq \tilde\varphi'(r)\leq 1,
	\\
	- 2(r\log^2(r))^{-1}&\leq \tilde\varphi^{\prime\prime}(r)\leq - (r\log^2(r))^{-1}.
\end{align*}

Similarly as in (i), we have 
\begin{align*}
&\left[ \frac{Lr_\circ^2}{9}\tilde\varphi^{\prime\prime}(\frac{r_\circ}{3} |\abar+tz|) -m_1 \sup_{B_2}|D^2\psi| \right] |z|^2\nonumber
\\
&\quad \leq \langle z, D^2_x \phi(\xbar+tz,\ybar) z\rangle \nonumber
\\
&\qquad\leq \left[12L\delta_0 |\abar|^{-1} + L\frac{r^2_\circ}{18} \tilde\varphi^{\prime\prime}(\frac{r_\circ}{3} |\abar+tz|) + m_1 \sup_{B_2}|D^2\psi| \right]|z|^2
\end{align*}
for $z\in\cone$. Since $|\abar|\leq \frac{1}{8}$, setting $\varepsilon$ small enough so that $\delta_0\in (0, \frac{1}{2}]$, we obtain
$$ 
2\log(|\abar|)\leq \log((1-\delta_0)|\abar|)\leq \log (|\abar + tz|)\leq \log((1+\delta_0)|\abar|)\leq \frac{1}{2}\log(|\abar|).
$$

Thus, for some constant $\kappa>0$, independent of $\varepsilon, L$, we have
\begin{align*}
&\left[ -L \kappa (|\abar|\log^2(|\abar|))^{-1}  -m_1 \sup_{B_2}|D^2\psi| \right] |z|^2\nonumber
\\
&\quad \leq \langle z, D^2_x \phi(\xbar+tz,\ybar) z\rangle \nonumber
\\
&\qquad\leq \left[12 L \varepsilon  (|\abar|\log^2(|\abar|))^{-1} - \frac{L}{\kappa} (|\abar|\log^2(|\abar|))^{-1} + m_1 \sup_{B_2}|D^2\psi| \right]|z|^2
\end{align*}
for $z\in\cone$. Now, first fixing $\varepsilon$ small enough depending on $\kappa^{-1}$, then choosing $L$ large we get
$$ -\kappa_1 L (|\abar|\log^2(|\abar|))^{-1} \leq \langle z, D^2_x \phi(\xbar+tz,\ybar) z\rangle\leq -\frac{1}{\kappa_1} L (|\abar|\log^2(|\abar|))^{-1}$$
for some $\kappa_1$. 

\smallskip

\noindent
$(iii)$ follows from the calculations in $(i)$ and $(ii)$ using the symmetry of $\varphi$ and considering $\psi=0$.
\end{proof}

We also required the following algebraic estimates from Section 3.1 in~\cite{KKL19}.
\begin{lem}
Suppose $p>1$ and $a, b\in \R$. Then 
\begin{equation}\label{EL1.6A}
\frac{1}{c_p} (|a|+|b|)^{p-2}\leq \int_0^1 |a+tb|^{p-2} dt\leq c_p (|a|+|b|)^{p-2},
\end{equation}
for some constant $c_p$, dependent only on $p$. 
\end{lem}

\section{The superquadratic case $p\geq 2$}\label{secsuper}

As we mentioned in our general strategy in subsection~\ref{subsecGS}, the aim is to provide useful estimates for $I_i, i=1,...,4$ in~\eqref{E1.4}
which would lead to a contradiction. Since they differ in the subquadratic and superquadratic cases, we concentrate in this section with the former one.

\subsection{Technical lemmas.} We provide a series of lemmas that will be used to estimate each $I_i, i=1,...,4$.
\begin{lem}[Estimate on $\cone$]\label{lemcone}
Let $p \geq 2$. Assume $(\bar x, \bar y)$ satisfies~\eqref{E1.3}. Then, there exist positive constants $c$ and $L_0$,
dependent on $p, \lambda, m_1, m$ and $\varepsilon$, where $\varepsilon$ is given by Lemma~\ref{L1.5}, such that
\begin{align*}
I_1 \geq & \ c\, (L \varphi'(|\bar a|))^{p-2} \int_{\cone} |z|^{p-2 - n - sp} \Theta_2 \phi(\cdot, \bar y)(\bar x, z) \dz,
\end{align*}
for all $L \geq L_0$.
\end{lem}

\begin{proof}
Note that $I_1$ is comprised of two terms. We only prove the estimate for the term involving $w_1$,  as the argument for the other term involving $w_2$ contributes the same an estimate following a similar argument. This is possible employing the inequality $\Phi(\bar x, \bar y + z) \leq \Phi(\bar x, \bar y)$ and noticing that the contribution of $\psi$ vanishes. 

Denote $\ell(z)=-\nabla_x\phi(\xbar, \ybar)\cdot z$. Since $\cone$ is symmetric (i.e. $z \in \cone$ iff $-z \in \cone$), and by the symmetry of $K$ (i.e. $K(z) = K(-z)$), we have
$$
\sL[\cone] \ell(\bar x) = \int_{\cone} |\ell(\bar x +z) - \ell(\bar x)|^{p-2} (\ell(\bar x + z)-\ell(\bar x)) K(z)dz =0.
$$
Using the inequality $\Phi(\xbar+z, \ybar)\leq \Phi(\xbar,\ybar)$, we get $u(\xbar+z)-u(\xbar)\leq \phi(\xbar+z,\ybar)-\phi(\xbar,\ybar)$, and then combining the
monotonicity of $J_p$ we have
$$
\sL[\cone]w_1(\bar x)\geq \sL[\cone]\phi(\cdot, \ybar)(\bar x).
$$
Since the real function $t \mapsto |t|^{p-2}t$ is absolutely continuous, we see that for all $a,b \in \R$ we have
\begin{equation}\label{intJp}
J_p(a) - J_p(b) = (p-1)\int_{0}^{1} |b + t(a - b)|^{p-2}(a - b)dt.
\end{equation}
Using this with $a=\phi(\xbar,\ybar)-\phi(\xbar+z, \ybar)$ and $b=-\ell(\xbar)+\ell(\xbar+z)=\ell(z)$, we obtain
\begin{align*}
\sL[\cone] w_1(\bar x) &=  \sL[\cone] w_1(\bar x) - \sL[\cone] \ell(\bar x) 
\\
&\geq  \sL[\cone]\phi(\cdot, \ybar)(\bar x)-\sL[\cone] \ell (\bar x) 
\\
&=  (p-1)\int_{\cone} \int_0^1|\ell(z) + t \Theta_2 \phi(\cdot, \bar y)(\bar x, z)|^{p-2} \Theta_2 \phi(\cdot, \bar y)(\bar x, z)dt\, K(z) \dz.
\end{align*}

Using \eqref{EL1.6A} together with the positivity of $\Theta_2 \phi(\cdot, \bar y)(\bar x, z)$ for $z \in \cone$, given for each case in Lemma~\ref{L1.5}, we conclude that
\begin{equation*}
\sL[\cone] w_1(\bar x) \geq \lambda\frac{(p-1)}{c_p} \int_{\cone}|\ell(z)|^{p-2}\Theta_2 \phi(\cdot, \bar y)(\bar x, z) \frac{\dz}{|z|^{n+sp}}.
\end{equation*}
Now, we see that
\begin{align*}
\nabla_x \phi(\bar x, \bar y) = L\varphi'(|\bar a|) \hat{\bar a} + m_1 \nabla \psi(\bar x).
\end{align*}
For $z \in \cone$, we also have $|\hat{\bar a} \cdot z| \geq (1 -  \delta_0)|z|$. Therefore
$$
|\ell(z)| \geq (L \varphi'(|\bar a|) (1 - \delta_0) -  m_1 \sup_{B_2}|\nabla \psi|)|z| \quad \mbox{for} \ z \in \cone.
$$
Since $\varphi'$ is away from zero for $t$ close to $0$, taking $L$ large enough we arrive at
$$
|\ell(z)| \geq \frac{L}{2} \varphi'(|\bar a|) |z| \quad \mbox{for} \ z \in \cone.
$$
Replacing this in the last estimate of $\sL[\cone] w_1(\bar x)$, we conclude that
\begin{align*}
\sL[\cone] w_1(\bar x) \geq \kappa (L \varphi'(|\bar a|))^{p-2} \int_{\cone} |z|^{p-2 - n - sp} \Theta_2 \phi(\cdot, \bar y)(\bar x, z) \dz. 
\end{align*}

An upper bound for $\sL[\cone] w_2(\bar y)$ is obtained in a similar fashion, namely
$$
\sL[\cone] w_2(\bar y) \leq -\kappa (L \varphi'(|\bar a|))^{p-2} \int_{\cone} |z|^{p-2 - n - sp} \Theta_2 \phi(\bar x, \cdot)(\bar y, z) \dz,
$$
and by Lemma~\ref{L1.5}(iii), we conclude the result.
\end{proof}

\begin{lem}[Estimate on $\mathcal D_1$]\label{lemD1}
Let $p \geq 2$. Assume $(\bar x, \bar y)$ satisfies~\eqref{E1.3} and set $\delta = \varepsilon_1 |\bar a|$ for some $\varepsilon_1 \in (0, \frac{1}{2})$. 
Then, there exist positive constants $C > 0$ and $L_0=L_{0}(m, m_1) $ such that
\begin{align*}
I_2 \geq & -C L^{p - 1} \varepsilon_1^{p(1 - s)} (\varphi'(|\bar a|))^{p-1} |\bar a|^{p(1 - s) - 1}
\end{align*}
for all $L \geq L_0$.
The constant $C > 0$ depends only on $n, p, s, \Lambda$, but not on $\bar a$ and $\varepsilon_1$.
\end{lem}

\begin{proof}
As done before, we write
\begin{align*}
\sL[\cD_1] w_1(\bar x) &= \sL[\cD_1] w_1(\bar x) - \sL[\cD_1] \ell (\bar x) 
\\
&= (p - 1) \int_{\cD_1} \int_{0}^{1}|\ell(z) + t \Theta_2 \phi(\cdot, \bar y)(\bar x, z)|^{p-2} \Theta_2 \phi(\cdot, \bar y)(\bar x, z)dt\, K(z) \dz
\end{align*}
For $z\in B_\delta$, we  have
\begin{align*}
\Theta_2 \phi(\cdot, \bar y)(\bar x, z) =  -\frac{1}{2}\int_{0}^{1}(1 - t) \langle D^2_x \phi(\bar x + tz) z, z \rangle dt.
\end{align*}
Observing that
\begin{align*}
\langle D^2_x \phi(\bar x + z, \bar y) z, z \rangle &=  L\Big{(}\varphi''(|\bar a + z|) - \frac{\varphi'(|\bar a + z|)}{|\bar a + z|} \Big{)} \langle \widehat{\bar a + z}, z \rangle^2 + L\frac{\varphi'(|\bar a + z|)}{|\bar a + z|} |z|^2 
\\
& + m_1 \langle D^2 \psi(\bar x + z) z, z \rangle,
\end{align*}
and taking into account that $|z| \leq \varepsilon_1 |\bar a|$ and  
$\varepsilon_1 < 1/2$, we get
$$
\langle D^2_x \phi(\bar x + tz) z, z \rangle \leq \Big{(} 2L\frac{\varphi'(|\bar a|/2)}{|\bar a|} + m_1 \sup_{B_{\delta}(\bar x)} |D^2 \psi| \Big{)}  |z|^2,
$$
using the fact that $\varphi$ is concave. 
Moreover, in every case of $\varphi$, since $\varphi'$ is uniformly bounded from below in a neighborhood of $0$, we have the existence of a universal constant $\kappa_0$ such that
$$
\langle D^2_x \phi(\bar x + tz) z, z \rangle \leq \kappa_0 L |\bar a|^{-1}\varphi'(|\bar a|)|z|^2, \quad z \in B_\delta,
$$
by taking $L$ large in terms of $m$ and $m_1$. Here we have used that $\varphi'(t)$ is proportional to $\varphi'(t/2)$ in every case of $\varphi$ for $t$ small.

By a similar argument, and since in every case the third derivative of $\varphi$ is positive, we arrive at
$$
\langle D^2_x \phi(\bar x + tz) z, z \rangle \geq \kappa_1 L (\varphi''(|\bar a|/2) - 2\frac{\varphi'(|\abar|/2)}{|\abar|})|z|^2- m_1 \sup_{B_{\delta}(\bar x)} |D^2 \psi||z|^2, \quad z \in B_\delta.
$$
Then, for some $\kappa_2 > 0$, taking $L$ large enough in terms of $m$ and $m_1$, these estimates imply that
\begin{equation}\label{Rio}
-\kappa_2 L |\bar a|^{-1}\varphi'(|\bar a|)|z|^2 \leq \Theta_2 \phi(\cdot, \bar y)(\bar x, z) \leq -\kappa_2 L (\varphi''(|\bar a|) - \frac{\varphi'(|\bar a|)}{|\abar|})|z|^2, \quad z \in B_\delta,
\end{equation}
where we have used that $\varphi''(t)$ is proportional to $\varphi''(t/2)$ for every $t$ small, in every case of $\varphi$ (see \eqref{thevarphis}).
Recalling $|\ell(z)| \leq (L \varphi'(|\bar a|) + m_1|\nabla \psi(\bar x)|)|z|$, for all $t \in (0,1)$ and $z \in B_\delta$, we have that
\begin{equation}\label{l+delta2}
|\ell(z)|+|\Theta_2 \phi(\cdot, \bar y)(\bar x, z))| \leq \kappa_3 L \Big{(} 1 + \varphi'(|\bar a|) + |\bar a| |\varphi''(|\bar a|)| \Big{)} |z|
\end{equation}
for some $\kappa_3 > 0$. Using the lower bound of $\Theta_2 \phi(\cdot, \bar y)(\bar x, z)$ given by \eqref{Rio}
together with \eqref{EL1.6A}, and then applying \eqref{l+delta2}, we arrive at
\begin{align*}
\sL[\cD_1] w_1(\bar x) &\geq  -\kappa_2  (p -1) L \frac{\varphi'(|\bar a|)}{|\bar a|} \int_{\cD_1} \Big{|} |\ell(z)| + |\Theta_2 \phi(\cdot, \bar y)(\bar x, z)| \Big{|}^{p-2} |z|^2 K(z)dz 
\\
&\geq  -\kappa_2 \kappa_3^{p - 2} (p-1) L^{p-1} \frac{\varphi'(|\bar a|)}{|\bar a|}(1 + |\varphi'(|\bar a|)| + |\bar a| |\varphi''(|\bar a|)|)^{p - 2} \int_{\cD_1} |z|^{p-2} |z|^2 K(z)dz 
\\
&\geq  -\kappa_4 \Lambda (p-1) L^{p-1} \frac{\varphi'(|\bar a|)}{|\bar a|}(1 + |\varphi'(|\bar a|)| + |\bar a| |\varphi''(|\bar a|)|)^{p-2} \int_{0}^{\varepsilon_1 |\bar a|} r^{p - ps -1}dr 
\\
&\geq  -\frac{\kappa_5 \Lambda (p - 1)}{p(1-s)}
L^{p - 1} (\varphi'(|\bar a|))^{p-1} \varepsilon_1^{p(1 - s)} |\bar a|^{p(1 - s) - 1}
\end{align*}
for some positive constants $\kappa_4, \kappa_5$. Here we have used that $|\bar a| |\varphi''(|\bar a|)| \leq \kappa_6\varphi'(|\bar a|)$ for some $\kappa_6 > 1$ and  $|\bar a|\leq \frac{1}{2}$,
 which can be easily seen to hold for the functions given by \eqref{thevarphis}.

Since a similar (upper) bound can be found for $\sL[\cD_1] w_2(\bar y)$ (with a slightly simpler computation), the result follows.
\end{proof}

\begin{lem}[Estimate on $\cD_2$]\label{lemD2}
Let $p \geq 2$. Assume $(\bar x, \bar y)$ satisfies~\eqref{E1.3}, and assume further that $u \in C^{0, \upkappa}(B_{\varrho_2})$ for some $\upkappa \geq 0$. Then, there exists $C =C(n, p, \Lambda, m, m_1)>0$
such that for every $L\geq L_0=L_0(\varrho_1, \varrho_2)>0$ we have
\begin{equation*}
I_3 \geq - C [u]_\upkappa^{p - 2} \Big{(}  \int_{\delta}^{\tilde \varrho} r^{\upkappa(p - 2) - sp + 1} dr + |\nabla \psi(\bar x)| \int_{\delta}^{\tilde \varrho} r^{\upkappa(p -2) - sp} dr \Big{)},
\end{equation*}
where $[u]_\upkappa=[u]_{C^{0,\upkappa}(B_{\varrho_2})}$ and $[u]_{C^{0,\upkappa}(A)}$ denotes the
$\upkappa$-H\"{o}lder seminorm on $A$ defined as
$$[u]_{C^{0,\upkappa}(A)}=\sup_{x\neq y, x,y\in A}\frac{|u(x)-u(y)|}{|x-y|^\upkappa}.$$
\end{lem}

\begin{proof}
By our choice of $L$ (that is, $L\varphi(\frac{\varrho_2-\varrho_1}{8})>2$),
 we have $|\abar|<\frac{1}{2}\tilde\varrho=\frac{\varrho_2-\varrho_1}{8}$, and since 
 $\xbar\in B_{\frac{\varrho_2+\varrho_1}{2}}$, $\ybar\in B_{\frac{5\varrho_2}{8}+\frac{3\varrho_1}{8}}$, we have $\xbar+z, \ybar+z\in B_{\varrho_2}\subset B_2$ for $z\in B_{\tilde\varrho}$.
Notice that using~\eqref{intJp}, and denoting 
\begin{equation}\label{theta1}
\Theta_1 f(x,z) := f(x) - f(x + z), 
\end{equation}
we have
 \begin{align*}
 I_3 = (p - 1) \int_{\cD_2} \int_{0}^{1} |\Theta_1 u(\bar y, z) + t(\Theta_1 u(\bar x, z) - \Theta_1 u(\bar y, z))|^{p-2} (\Theta_1 u(\bar x, z) - \Theta_1 u(\bar y, z)) dt\, K(z)\dz
 \end{align*}
Using  $\Phi(\xbar,\ybar)\geq \Phi(\xbar+z,\ybar+z)$, we note that
$$
u(\xbar)-u(\xbar+z)\geq u(\ybar)-u(\ybar+z) + m_1 (\psi(\xbar)-\psi(\xbar+z)),
$$
from which, we get $\Theta_1 u(\bar x, z) - \Theta_1 u(\bar y, z) \geq m_1 \Theta_1 \psi(\bar x, z)$, and we can write
\begin{align}\label{I3superback}
 I_3 \geq -m_1(p - 1) \int_{\cD_2} \int_{0}^{1} |\Theta_1 u(\bar y, z) + t(\Theta_1 u(\bar x, z) - \Theta_1 u(\bar y, z))|^{p-2}\,  dt\,  |\Theta_1 \psi(\bar x, z)| K(z)\dz .
\end{align}
At this point, we recall \eqref{theta2} to note that 
$$
|\Theta_1 \psi(\xbar, z)|\leq |\Theta_2 \psi(\bar x, z)| + |\nabla\psi(\xbar)||z|\leq \kappa(|z|^2+ |\nabla \psi(\bar x)||z|)
$$
for some constant $\kappa$. 
Moreover, using the $\upkappa$-H\"{o}lder regularity of $u$, we have
$$
|\Theta_1 u(\bar y, z) + t(\Theta_1 u(\bar x, z) - \Theta_1 u(\bar y, z))| \leq 3 [u]_{\upkappa} |z|^{\upkappa}.
$$
Thus, we arrive at
\begin{align}\label{new}
I_3 
&\geq -(p-1) \kappa_1 [u]_\upkappa^{p-2} \int_{\cD_2} |z|^{\upkappa (p - 2)}  (|z|^2+ |\nabla \psi(\bar x)||z|) K(z)dz \nonumber
\\
&\geq -\kappa_2 [u]_\upkappa^{p-2}  \Big{(}\int_{\cD_2} |z|^{\upkappa(p-2) + 2 - n - sp} dz + |\nabla \psi(\bar x)| \int_{\cD_2} |z|^{\upkappa(p - 2) + 1 - n - sp}dz \Big{)}, 
\end{align}
where $\kappa_2=\kappa_2(p, \Lambda, m, m_1)$ is some constant.
Since $\cD_2 \subset B_{\tilde\varrho} \setminus B_\delta$, we readily have
\begin{equation*}
I_3 \geq - \kappa_3 [u]_\upkappa^{p - 2} \Big{(}  \int_{\delta}^{\tilde \varrho} r^{\upkappa(p - 2) - sp + 1} dr + |\nabla \psi(\bar x)| \int_{\delta}^{\tilde \varrho} r^{\upkappa(p -2) - sp} dr \Big{)}
\end{equation*}
for some constant $\kappa_3=\kappa_3(n, p, s, \Lambda, m, m_1)$,
from which the result follows.
\end{proof}

Finally, the following estimate on $B_{\tilde\varrho}^c$ can be easily derived from~\eqref{E1.1}.
\begin{lem}\label{lemI4}
Let $p > 1$. There exists $\kappa_4 > 0$ just depending on the data such that
$$ 
|I_4|\leq \kappa_4(\sup_{B_2}|u|^{p-1}+ (\tail(u; 0, 2))^{p-1})\leq 2\kappa_4.
$$
\end{lem}


\subsection{Proof of Theorem~\ref{Tmain-1} in the superquadratic case.}
Now we are ready to provide the regularity estimates in the superquadratic case. We start with the following self-improving result. Furthermore, the proof of this proposition
contains all the key estimates required for the proof of Theorem~\ref{Tmain-1}. We also use the following notation to express the dependence on the 
constants.
\begin{align*}
\data &= (n, p, s, \lambda, \Lambda),
\\
\dataex &= (n, p, s, \lambda, \Lambda, m, m_1, \varrho_1,\varrho_2).
\end{align*}

\begin{prop}\label{T1.7}
Let $p \geq 2$ and $s\in (0, 1)$.
Let $\gamma_\circ=\gamma_\circ(s)=\min\{1, \frac{ps}{p-1}\}$.  Consider $1\leq \varrho_1<\varrho_2\leq 2$. Suppose that
$u\in C^{0,\upkappa}(B_{\varrho_2})\cap C(\bar{B}_2)\cap L^{p-1}_{sp}(\Rn)$ solves \eqref{E1.1} for some $\upkappa\in[0,\gamma_\circ)$.
Then for any $\gamma<\min\{\gamma_\circ, \upkappa+\frac{1}{p-1}\}$, there exists $L_\gamma$ such that
$$
\norm{u}_{C^{0,\gamma}(B_{\varrho_1})}\leq L_\gamma,
$$
and $L_\gamma$ depends on the bound of $[u]_{C^{0,\upkappa}(B_{\varrho_2})}$, $\dataex$, $\upkappa$ and $\gamma$. 
\end{prop}

\begin{proof}
Since $|u|\leq 1$ in $B_2$, we only estimate the $\gamma$-H\"{o}lder seminorm of $u$ in $B_{\varrho_1}$.

Fix $\gamma\in (0, \min\{\gamma_\circ(s), \upkappa+\frac{1}{p-1}\})$. Set $m$ large enough so that
\begin{equation}\label{ET1.70}
\upkappa (p-2) + \frac{m-1}{m}\upkappa + 1>\gamma(p-1),\quad \text{and}\quad \gamma_\circ(s)>\frac{m-1}{m}\upkappa\geq \frac{\upkappa}{2}.
\end{equation}
We consider $\Phi$ in \eqref{E1.2} with the function $\varphi=\varphi_\gamma$,
that is,
$$
\Phi(x, y)=u(x)-u(y)-L\varphi_\gamma(|x-y|)-m_1\psi(x)\quad x, y\in B_2.
$$
As mentioned above, for $L\varphi(\frac{\varrho_2-\varrho_1}{8})>2$, we have a maximizer $(\xbar, \ybar)$. Since 
$\Phi(\xbar,\ybar)>0$, we have $L\varphi_\gamma(|\abar|)\leq u(\xbar)-u(\ybar)\leq 2$. We set $L> 2 [\frac{\varrho_2-\varrho_1}{8}]^{-\gamma}$. Then
$|\abar|<\frac{\varrho_2-\varrho_1}{8}$.
From $\Phi(\xbar,\ybar)>0$, it also follows that
$$\psi(\xbar) \leq \frac{1}{m_1} (u(\xbar)-u(\ybar))\leq \frac{1}{m_1} [u]_\upkappa\, |\abar|^\upkappa,
\quad \text{where}\quad [u]_\upkappa:=[u]_{C^{0,\upkappa}(B_{\varrho_2})},
$$
giving us $(|\xbar|^2-\varrho_1^2)_+\leq \left(\frac{1}{m_1}[u]_\upkappa |\abar|^\upkappa\right)^{\frac{1}{m}}$. Hence
\begin{equation}\label{ET1.7A}
|\nabla\psi(\xbar)|\leq 4m \left(\frac{1}{m_1} [u]_\upkappa |\abar|^\upkappa\right)^{\frac{m-1}{m}}.
\end{equation}

Now we estimate the terms $I_i, i=1,2,3,4$, from \eqref{E1.4}. Set $\delta_0=\varepsilon\wedge \frac{1}{2}$ where
$\varepsilon$ is given by Lemma~\ref{L1.5}.
For $I_1$, using the above choice of $\varphi$ in 
Lemma~\ref{L1.5}(i) and Lemma~\ref{lemcone}, we conclude that
\begin{equation*}
I_1 \geq c_1  L^{p-1} |\bar a|^{(\gamma - 1)(p - 2) + \gamma - 2} \int_{\cone} |z|^{p  - n - sp} \dz
\end{equation*}
for all $L\geq L_0=L_0(p, \lambda, m, m_1)$.
The estimates in \cite[Example~1]{BCCI12} give
\begin{equation}\label{intconeBarles}
\int_\cone |z|^{p-n-ps} \dz= \frac{{\rm Vol}(\cone)}{{\rm Vol}(B_{\delta_0 |\abar|})}\int_{B_{\delta_0 |\abar|}} |z|^{p-n-ps} \dz\gtrsim 
\frac{1}{p(1-s)}\delta_0^{\frac{n-1}{2}} (\delta_0|\abar|)^{p-ps}.
\end{equation}
Combining these estimates we arrive at 
\begin{equation}\label{ET1.7B}
I_1\geq  C_{\varepsilon} L^{p-1}|\abar|^{\gamma(p-1)-ps}
\end{equation}
for all $L\geq L_0$ and $C_\varepsilon=C_\varepsilon(n, p, s, \varepsilon)$.
For $I_2$, we employ Lemma~\ref{lemD1} with this choice of $\varphi$. Since $\gamma < 1$ and $|\bar a|$ is small when $L$ is large, we can write
\begin{equation*}
I_2\geq -C L^{p - 1} \varepsilon_1^{p(1 - s)} |\bar a|^{\gamma(p-1)-ps} \quad \text{for all}\; L\geq L_0=L_0(m, m_1).
\end{equation*}
In view of \eqref{ET1.7B},  we can choose
$\varepsilon_1\in (0, \varepsilon)$ small enough, depending on $\varepsilon$ and not on $L$ and $|\abar|$, so that
\begin{equation}\label{ET1.7D}
I_1+I_2\geq \frac{C_\varepsilon}{2} L^{p-1}|\abar|^{\gamma(p-1)-ps}
\end{equation}
for $L\geq L_0$ where $L_0$ depends on $n, p, s, \lambda, m$ and $m_1$.

\medskip

Now we fix the choice of $\varepsilon$ and $\varepsilon_1$ from above.
For $I_3$, we use Lemma~\ref{lemD2} together with~\eqref{ET1.7A} to conclude that
\begin{equation}\label{rev1}
I_3 \geq -C  [u]_{\upkappa}^{p-2}\Big{(} \underbrace{\int_{\delta}^{1} r^{\upkappa(p - 2) - sp + 1}\, dr}_{=\mathcal{I}} + [u]_{\upkappa}^{\frac{m - 1}{m}} |\bar a|^{\frac{(m - 1)\upkappa}{m}} \int_{\delta}^{1} r^{\upkappa(p - 2) - sp}\, dr \Big{)}
\end{equation}
for all $L\geq L_0(\varrho_1, \varrho_2)$.
 To calculate the second integration we recall that (by our choice)
$$\upkappa(p-2)\geq \gamma(p-1) -\frac{m-1}{m}\upkappa-1,
\quad \text{and}\quad \frac{m-1}{m}\upkappa\geq \frac{\upkappa}{2}\, .$$

Therefore, since $\gamma(p-1) -\frac{m-1}{m}\upkappa-1 - sp<0$, we get
\begin{align*}
|\abar|^{\frac{(m - 1)\upkappa}{m}} \int_{\delta}^{1} r^{\upkappa(p - 2) - sp}\, dr &\leq
|\abar|^{\frac{(m - 1)\upkappa}{m}} \int_{\delta}^{1} r^{\gamma(p-1) -\frac{m-1}{m}\upkappa-1 - sp}\, dr
\\
&\leq \frac{|\abar|^{\frac{(m - 1)\upkappa}{m}}}{sp+\frac{m-1}{m}\upkappa-\gamma(p-1)} 
\delta^{\gamma(p-1) -\frac{m-1}{m}\upkappa - sp}
\\
&\leq \frac{1}{sp+\upkappa/2-\gamma(p-1)} \varepsilon_1^{\gamma(p-1) -\frac{m-1}{m}\upkappa - sp}
|\abar|^{\gamma(p-1)-sp}.
\end{align*}

The integral $\mathcal{I}$ is computed as follows.
\begin{enumerate}
\item If $\upkappa(p-2)+2 -sp>0$, we have $\mathcal{I}\leq C_\upkappa$ for some constant $C_\upkappa$.
\item If $\upkappa(p-2)+2 -sp=0$, then 
$$\mathcal{I}\leq -\log (\varepsilon_1|\abar|)\leq (|\log\varepsilon_1| + C |\abar|^{\gamma(p-1)-sp})$$
for some constant $C$, using the fact that $\gamma(p-1)-sp<0$.
\item If $\upkappa(p-2)+2 -sp<0$, then using the fact that for $x\in [1, \infty)$, the map $(0, \infty)\ni\alpha\mapsto\frac{x^\alpha-1}{\alpha}$ is
increasing, we obtain
$$\mathcal{I}\leq  \frac{\delta^{\upkappa(p-2)+2 -sp}-1}{sp-2-\upkappa(p-2)}\leq 
 \frac{\delta^{\gamma(p-1)-sp}-1}{sp-\gamma(p-1)}\leq C_\upkappa \varepsilon_1^{\gamma(p-1)-sp}|\abar|^{\gamma(p-1)-sp}$$
using $0< sp-2-\upkappa(p-2)<sp- \gamma(p-1)$ (see \eqref{ET1.70}),  where $C_\upkappa=\frac{1}{sp-\gamma(p-1)}$.
\end{enumerate}
Hence, combining the above estimates we get
\begin{equation}\label{ET1.7E}
I_3 \geq - C_{\upkappa,\varepsilon_1} |\bar a|^{\gamma(p - 1) - sp}
\end{equation}
for some $ C_{\upkappa,\varepsilon_1}$ dependent on 
$[u]_\upkappa, \varepsilon_1$, but not on $L$ and $|\bar a|$.
The bound for $I_4$ comes immediately from Lemma~\ref{lemI4}. Gathering these estimates with~\eqref{ET1.7D} and \eqref{ET1.7E} into \eqref{E1.4}, 
 we obtain
\begin{equation*}
(\frac{C_\varepsilon}{2} L^{p-1} - C_{\upkappa, \varepsilon_1}) |\bar a|^{\gamma(p - 1) - sp} \leq 2 + 2 \kappa_4
\end{equation*}
for all $L\geq L_0$, where $L_0=L_0(\dataex)$ is a constant.
Since $\gamma < \gamma_\circ$, the exponent of $|\bar a|$ is negative, 
and the last inequality clearly cannot hold for large enough $L$, leading to a contradiction. 
Hence, if we set $L=L_\gamma>L_0$ large enough so that
the display can not hold, we must have $\Phi\leq 0$ in $B_2\times B_2$. Hence
$$u(x)-u(y)\leq L_\gamma\varphi_\gamma(|x-y|)=L_\gamma |x-y|^{\gamma}\quad x, y\in B_{\varrho_1}.$$
This completes the proof.
\end{proof}

Now we are ready to provide the proof of Theorem~\ref{Tmain-1} in the superquadratic case. 
\begin{proof}[Proof of Theorem~\ref{Tmain-1} (superquadratic case)]
As discussed above, it is enough to consider \eqref{E1.1}. 
We first establish that $u$ is $C^{0,\gamma}$ for $\gamma$ smaller than, but arbitrarily close to $\gamma_\circ(s)$. 

Fix $\gamma\in (0, \gamma_\circ(s))$ and let $\alpha= \frac{1}{2}\min\{\gamma, \frac{1}{p-1}\}$. Let $k\in\mathbb{N}$ be such that $k \alpha <\gamma \leq (k+1)\alpha$. We set $\upkappa_0=0$ and $\upkappa_i=i\alpha$. Note that $\upkappa_i<\upkappa_{i+1}< \upkappa_i \alpha +\frac{1}{p-1}$ for $i=0, \ldots, k$, $\upkappa_k<\gamma<\gamma_\circ(s)$, giving us
$\upkappa_{i+1}<\min\{\gamma_\circ(s), \upkappa_i + \frac{1}{p-1}\}$ for $i=0,\ldots, k-1$.
By this choice we
also have 
$$
sp-\upkappa_i(p-1)+\frac{1}{2}\upkappa_{i-1}\geq \min\left\{\frac{\gamma(p-1)}{2}, \frac{\alpha}{2}\right\}\quad \text{for}\; i=1, \ldots, k.
$$
Now choose
a sequence of strictly decreasing numbers $2>\varrho_0>\varrho_1>\ldots>\varrho_k>1$. Applying Proposition~\ref{T1.7} successively
we see that $u\in C^{0, \upkappa_i}(B_{\varrho_i})$ and the norm $\norm{u}_{C^{0, \upkappa_i}(B_{\varrho_i})}$ is bounded by a universal
constant dependent on $\data$. 
Again, applying Proposition~\ref{T1.7} with $\upkappa=\upkappa_k$ (since $\gamma<\min\{\gamma_\circ(s), \upkappa_k+\frac{1}{p-1}\}$), we establish that $u\in C^{0,\gamma}(B_1)$
and $\norm{u}_{C^{0, \gamma}(B_{1})}\leq C_1$ for some constant $C_1$.

\medskip

Next, we suppose that $\gamma_\circ=\frac{sp}{p-1}<1$. We show that $u\in C^{0, \gamma_\circ}(B_1)$. Since
$$ 
\frac{sp}{p-1}(p-2) -sp +1=1-\frac{sp}{p-1}>0,
$$
we can choose $\upkappa<\gamma_\circ$ large enough
so that $\upkappa(p-2)+2-ps=(\upkappa-\gamma_\circ)(p-2) + 2-\gamma_\circ>0$, $\upkappa(p-2)-sp +1>0$, and $\gamma_\circ-\upkappa<\frac{1}{p-1}$. From the first part we know that 
$u\in C^{0, \upkappa}(B_{3/2})$ and $\norm{u}_{C^{0,\upkappa}(B_{3/2})}\leq C_2$ for some $C_2>0$. 

Then, we employ the estimates presented in Proposition~\ref{T1.7} with $\varphi = \varphi_{\gamma_\circ}$ and $\varrho_1=1<\varrho_2=3/2$.
Also, fix any $m>3$. In this case, the estimate for $I_1 + I_2$ from~\eqref{ET1.7D} reads
$$
I_1+I_2\geq \frac{1}{2}C_{\varepsilon} L^{p-1}
$$
for all $L\geq L_0=L_0(\dataex)$. 
 By our choice of $\upkappa$, we see from \eqref{rev1} that
$$I_3\geq -C$$
for some constant $C=C(n,p,s,\Lambda)$ and $L\geq L_0$, where $L_0$ is chosen from Lemma~\ref{lemD2}. Since the estimate for $I_4$ is the same as before (Lemma~\ref{lemI4}), we arrive at
$$
 \frac{1}{2}C_{\varepsilon} L^{p-1}\leq 2 + \kappa(\varepsilon_1) + 2\kappa_4,
$$
for all $L\geq L_0=L_0(\dataex)$. Again, if we fix any $L>L_0$, which violates the above inequality, we must have $\Phi\leq 0$ in $B_2\times B_2$, implying the result.

\medskip

Now we are left with the case $\frac{sp}{p-1}>1$ and $\gamma_\circ=1$. Recall $\tilde\varphi$ from Lemma~\ref{L1.5} that satisfies
the following:
\begin{equation}\label{logchoice}
\begin{gathered}
	\frac{r}{2}  \leq \tilde\varphi(r)\leq r,\quad
	\frac{1}{2}  \leq \tilde\varphi'(r)\leq 1,
	\\
	- 2(r\log^2(r))^{-1}\leq \tilde\varphi^{\prime\prime}(r)\leq - (r\log^2(r))^{-1}.
\end{gathered}
\end{equation}
for  all $r\in (0, r_\circ]$. We let 
$\varphi(t)=\tilde\varphi(\frac{r_\circ}{3}t)$ for $t\in [0, 2]$
This time we set the parameters as follows
\begin{equation}
	\begin{gathered}
			\phi(x, y)=L\varphi(|x-y|)+m_1\psi(x), \quad \delta_0=\varepsilon (\log^2(|\abar|))^{-1}, \quad \rho=\frac{\frac{n+1}{2}+p-sp}{p-sp},
			\\
			\delta=\varepsilon_1 |\abar| (\log^{2\rho}(|\abar|))^{-1}, \quad \cone=\{z\in B_{\delta_0|\abar|} \; :\; |\langle \abar, z\rangle| \geq (1-\delta_0)|\abar||z|\},\label{defrho}
            \\
			\cD_1=B_\delta\cap \cone^c\quad \cD_2= B_{\frac{1}{8}}\setminus(\cD_1\cup\cone).
	\end{gathered}
\end{equation}
Let $m \geq 3$ be large and $\upkappa \in (0,1)$ be close enough to $1$ so that
\begin{equation*}
\frac{m - 1}{m} \upkappa + \upkappa(p-2) + 1-sp > 0.
\end{equation*}
From the first part we know that $u\in C^{0, \upkappa}(B_{\frac{3}{2}})$
and 
$$\norm{u}_{C^{0,\upkappa}(B_{3/2})}\leq C_2$$
 for some constant $C_2=C_2(\data)>0$. We consider the doubling function
 $$\Phi(x, y)= u(x)-u(y)-L \varphi(|x-y|)- m_1 \psi(x)\quad x, y\in B_2,$$
 where $\psi(x) = [(|x|^2-1)_+]^m$.
  Now choose $m_1$ large enough so that $m_1\psi(x)\geq 4$ for $|x|\geq 5/4$.
 From now on we fix this choice of $m$ and $m_1$.
 Suppose that, for all large $L$, we have~\eqref{E1.3}.

As before, set $\abar=\xbar-\ybar$ and note that  $|\xbar|\leq \frac{5}{4}$.
Again, since $L \varphi(t)\geq L \varphi(\frac{1}{8})\geq \frac{L}{16}$ for $t\in[ \frac{1}{8}, 2]$,  for large $L$ we must have $|\abar|< \frac{1}{8}$.
Moreover,
\begin{equation*}
\frac{1}{2} |\abar| L\leq L \varphi(|\abar|)\leq u(\xbar)-u(\ybar)\leq 2\; \Rightarrow L |\abar|\leq 4,
\end{equation*}
and \eqref{ET1.7A} becomes
\begin{equation}\label{ET1.7G}
|\nabla\psi(\xbar)|\leq \kappa_1 |\abar|^{\frac{m-1}{m}\upkappa},
\end{equation}
for some constant $\kappa_1=\kappa_1(m, C_2)$.
Using Lemma~\ref{lemcone} and Lemma~\ref{L1.5}(ii)-(iii), we see that
\begin{align*}
I_1 \geq \kappa L^{p-1} (|\abar|\log^2(|\abar|))^{-1} \int_{\cone} |z|^{p - n - sp} dz  = \kappa L^{p-1} (|\abar|\log^2(|\abar|))^{-1} \delta_0^{(n - 1)/2} (\delta_0 |\bar a|)^{p - sp},
\end{align*}
using \eqref{intconeBarles},
and by the above choice of $\delta_0$ and the  we arrive at
\begin{equation}\label{ET1.4H}
I_1\geq C_{\varepsilon} L^{p-1}|\abar|^{p-1-ps} 
(\log^2(|\abar|))^{-\beta},
\end{equation}
for all $L\geq L_0$, where we have denoted $\beta = \frac{n+1}{2}  + p-sp$ and the choice of $\varepsilon$ is set from Lemma~\ref{L1.5}.

Notice that by~\eqref{defrho} we have $\rho = \frac{\beta}{p(1 - s)}$.
Using Lemma~\ref{lemD1} and the choice of $\delta$ above, we see that
\begin{align*}
I_2 \geq -c L^{p-1} (\varepsilon_1 \log^{-2\rho}(|\abar|))^{p(1 - s)}|\bar a|^{p - sp - 1} = -c L^{p-1} \varepsilon_1^{p(1 - s)} |\bar a|^{p - sp - 1 } (\log^2(|\abar|))^{-\beta},
\end{align*}
from which, taking $\varepsilon_1$ small enough with respect to  $C_{\varepsilon}$ in~\eqref{ET1.4H} we obtain that
\begin{equation}\label{ET1.4I}
I_1+I_2\geq \frac{C_{\varepsilon}}{2}L^{p-1} 
|\abar|^{p-ps-1}(\log^2(|\abar|))^{-\beta},
\end{equation}
for all $L\geq L_0$. 
The estimate for $I_4$ is exactly as before, see Lemma~\ref{lemI4}.

For $I_3$, we use Lemma~\ref{lemD2} together with~\eqref{ET1.7G} to obtain
$$I_3 \geq - \kappa\left(  \int_{\delta}^{1} r^{\upkappa(p - 2) - sp + 1} dr + |\abar|^{\frac{m-1}{m}\upkappa} \int_{\delta}^{1} r^{\upkappa(p -2) - sp}\, dr \right),
$$
where the constant $\kappa$ depends on $C_2, m, m_1$. Now by our choice of $\upkappa$ we have 
$$\upkappa (p-2) -sp+2>  \upkappa (p-2) -sp+1 +\frac{m-1}{m}\upkappa>0,$$
giving us,
$$\int_{\delta}^{1} r^{\upkappa(p - 2) - sp + 1} dr\leq 
\frac{1}{\upkappa (p-2) -sp+2}.$$
Again, since $\upkappa(p-2)-sp\neq -1$,
\begin{align*}
|\abar|^{\frac{m-1}{m}\upkappa} \int_{\delta}^{1} r^{\upkappa(p -2) - sp}\, dr &\leq \frac{1}{sp-\upkappa(p-2)-1} 
|\abar|^{\frac{m-1}{m}\upkappa} \delta^{\upkappa(p -2) - sp + 1}
\\
&\leq \frac{1}{sp-(p-1)} \varepsilon_1^{\upkappa(p -2) - sp + 1} 
|\abar|^{\sigma} (\log^{2\rho}(|\abar|))^{sp-\upkappa(p -2)-1}
\end{align*}
where $\sigma=\upkappa(p -2) - sp + 1+\frac{m-1}{m}\upkappa>0$.
Thus, for some constants $\kappa(\varepsilon_1), \kappa_1$, we have
\begin{equation}\label{ET1.721}
I_3\geq -\kappa_1 -\kappa(\varepsilon_1) |\abar|^{\sigma} (\log^{2\rho}(|\abar|))^{sp-\upkappa(p -2)-1}.
\end{equation}
Now, combining \eqref{ET1.4I} and \eqref{ET1.721} we obtain
$$\frac{C_{\varepsilon}}{2}L^{p-1} 
|\abar|^{p-ps-1}(\log^2(|\abar|))^{-\beta} - \kappa(\varepsilon_1) |\abar|^{\sigma} (\log^{2\rho}(|\abar|))^{sp-\upkappa(p -2)-1}
\leq 2 + \kappa_1 + \kappa_4,$$
which can not hold for large enough $L$. Thus, we arrive at a contradiction in the same way as before, completing the proof.

\end{proof}


\section{The subquadratic case $1<p<2$}\label{S-sub}
The argument in this case is broadly similar to the superquadratic case, but some of the key estimates change due to the singularity appearing
within the regime $p < 2$. 

\subsection{Technical lemmas.} We recall the definition of $\phi$ in~\eqref{defphi}, and let us start with the estimate on the cone $\cone$.
\begin{lem}(Estimate on $\cone$)\label{lemconesub}
Let $1 < p < 2$. Assume $(\bar x, \bar y)$ satisfies~\eqref{E1.3}. Then, there exist positive constants $c$ and $L_0$,
dependent on $m_1, m, \lambda, p$,
such that
\begin{align*}
I_1 \geq & \ c  (L \varphi'(|\bar a|))^{p-2} \int_{\cone} |z|^{p-2 - n - sp}\, \Theta_2 \phi(\cdot, \bar y)(\bar x, z) \dz
\end{align*}
for all $L \geq L_0.$
\end{lem}

\begin{proof}
Exactly as in the proof of Lemma~\ref{lemcone}, we have
\begin{equation*}
\sL[\cone] w_1(\bar x) \geq
(p-1)\int_{\cone} \left[\int_0^1|\ell(z) + t \Theta_2 \phi(\cdot, \bar y)(\bar x, z)|^{p-2} dt\right]  \Theta_2 \phi(\cdot, \bar y)(\bar x, z) K(z) \dz.
\end{equation*}
Using~\eqref{l+delta2} (which is valid for $z$ small enough, independent of $|\bar a|$), we have
$$
|\ell(z) + t \Theta_2 \phi(\cdot, \bar y)(\bar x, z)| \leq L \kappa (1 + \varphi'(|\bar a|)) |z| + |\bar a| \varphi''(|\bar a|) |z|
$$
for some $\kappa > 0$ not depending on $|\bar a|$ and $L$. Since for all the cases of $\varphi$ (see \eqref{thevarphis}) we have $|\bar a| |\varphi''(|\bar a|)| \leq \kappa_1 \varphi'(|\bar a|)$ for some universal constant $\kappa_1$ and $|\abar|\leq \frac{1}{2}$, and taking into account the positivity of $\Theta_2 \phi(\cdot, \bar y)(\bar x, z)$ for $z \in \cone$ together with the first inequality in~\eqref{EL1.6A}, we conclude that
$$
\sL[\cone] w_1(\bar x) \geq \kappa_2 (L\varphi'(|\bar a|))^{p-2} \int_{\cone} |z|^{p - 2 - n - sp} \Theta_2 \phi(\cdot, \bar y)(\bar x, z)\dz,
$$
using the fact that $L\varphi'(|\abar|)\geq 1$ for all $L\geq L_0$ for some $L_0$ suitably fixed at a large value.
The result follows by the same arguments as at the end of the proof of Lemma~\ref{lemcone}.
\end{proof}

\begin{lem}[Estimate on $\mathcal D_1$]\label{lemD1sub}
Let $1 < p < 2$. Assume $(\bar x, \bar y)$ satisfies~\eqref{E1.3} and set $\delta = \varepsilon_1 |\bar a|$ for  $\varepsilon_1 \in (0, \frac{1}{2})$. Then, there exists $C=C(n, p, s,\Lambda)$ such that
\begin{align*}
I_2 \geq & -C \varepsilon_1^{p(1 - s)} (L\varphi'(|\bar a|))^{p-1} |\bar a|^{p(1 - s)-1}
\end{align*}
for all $L \geq L_0$, where $L_0=L_0(m, m_1)$.
\end{lem}

\begin{proof}
As in Lemma~\ref{lemD1}, we write
\begin{align*}
\sL[\cD_1] w_1(\bar x) = &\sL[\cD_1] w_1(\bar x) - \sL[\cD_1] \ell (\bar x) \\
= & (p - 1) \int_{\cD_1} \int_{0}^{1}|\ell(z) + t \Theta_2 \phi(\cdot, \bar y)(\bar x, z)|^{p-2} \Theta_2 \phi(\cdot, \bar y)(\bar x, z)dt\, K(z) \dz.
\end{align*}
Using \eqref{Rio} we have
$$|\Theta_2 \phi(\cdot, \bar y)(\bar x, z)|\leq C L |\abar|^{-1}\varphi^\prime(|\abar|)|z|^2:=\tau |z|^2 \quad \text{for all}\; z\in B_\delta.$$
This estimates holds for both the functions in \eqref{thevarphis}.
Now using the ellipticity condition on $K$ and \eqref{EL1.6A}, we arrive at
\begin{align*}
\sL[\cD_1] w_1(\bar x) \geq &  - C  \int_{B_\delta} 
(|\ell(z)| + \tau |z|^2)^{p-2} dt\, \tau |z|^{-n - ps + 2} \dz,
\end{align*}
where we use the fact that $[0, \infty)\ni t\mapsto (a+t)^{p-2}t$ is non-decreasing for $a\geq 0$. Denote by $\xi=\nabla_x\phi(\xbar,\ybar)$ and
we note that $\frac{L}{2} \varphi^\prime(|\abar|)\leq |\xi|\leq 2L \varphi^\prime(|\abar|)$ for all $L$ large. Denote by $\hat\xi=\xi/|\xi|$.
With these notation we write the above as
\begin{align*}
\sL[\cD_1] w_1(\bar x) \geq &  -C \tau |\xi|^{p-2} \int_{0}^\delta
\int_{\mathbb S^{n-1}} \Big{(} |\hat\xi\cdot \omega| + \frac{\tau}{|\xi|}r \Big{)}^{p-2}  r^{p-sp-1} {\rm d}\omega\, \dr
 \\
&\geq   -C \tau |\xi|^{p-2} \int_{0}^\delta \Big{(} 1 + \frac{\tau}{|\xi|}r \Big{)}^{p-2}  r^{p-sp-1} \, \dr
\\
&\geq -C \tau |\xi|^{p-2} \int_{0}^\delta r^{p-sp-1} \, \dr
\\
&=- C \tau |\xi|^{p-2} \delta^{p-sp}
\end{align*}
where in the second line we used \cite[Lemma~3.5]{KKL19}. Now, inserting the value of $\tau$ and the lower bound of $|\xi|$ we get
\begin{align*}
\sL[\cD_1] w_1(\bar x) \geq  -\kappa_2 \varepsilon_1^{p(1 - s)}(L \varphi'(|\bar a|))^{p-1} |\bar a|^{p(1 - s) - 1}.
\end{align*}
Since a similar lower bound can be obtained for $-\sL[\cD_1] w_2(\bar y)$, we conclude the result.
\end{proof}

\begin{lem}\label{lemD2sub}
Let $1 < p < 2$. Assume $(\bar x, \bar y)$ satisfies~\eqref{E1.3} and set $\delta = \varepsilon_1 |\bar a|$ for $\varepsilon_1 \in (0, \frac{1}{2})$. Then, there exist positive constants $\kappa=\kappa(n,p,m,m_1)$ and $L_0 =L_0(\varrho_1, \varrho_2)$ such that
$$
I_3 \geq -\kappa  \Big{(} \int_{\delta}^{\tilde \varrho} r^{2(p-1) - sp - 1} dr + |\nabla \psi(\bar x)|^{p-1} \int_{\delta}^{\tilde \varrho} r^{p-1 - sp - 1}dr\Big{)}
$$
for all $L \geq L_0$.
\end{lem}

\begin{proof}
We start by claiming that
$$
|J_p(a)-J_p(b)|\leq 2 |a-b|^{p-1}\quad a, b\in\R.
$$
In fact, when $a, b$ has the same sign, it is easy to verify. For $a\geq 0\geq b$, we have $a-b\geq a$ and $a-b\geq -b$. Therefore,
$$
J_p(a)-J_p(b)=J_p(a)+J_p(-b)\leq 2J_p(a-b),
$$
using the monotonicity of $J_p$. This proves the claim.

Notice that
$$
I_3 = \int_{\cD_2} \Big{(}J_p(\Theta_1 u(\bar x, z)) - J_p(\Theta_1 u(\bar y, z)) \Big{)} K(z) \dz
$$
Recall that for $|z|\leq \frac{\varrho_2-\varrho_1}{8}$ we have $\Phi(\xbar+z,\ybar+z)\leq\Phi(\xbar,\ybar)$, implying that
$$
\Theta_1 u(\bar y, z) + m_1 \Theta_1 \psi(\xbar, z) \leq \Theta_1 u(\xbar, z),
$$
 where $\Theta_1$ is given by \eqref{theta1}.
Therefore, using the monotonicity of $J_p$, we get
\begin{equation}\label{Rio2}
	\begin{split}
I_3 \geq & \int_{\cD_2} \Big{(}J_p(\Theta_1 u(\bar y, z) + m_1 \Theta_1 \psi(\bar x, z)) - J_p(\Theta_1 u(\bar y, z)) \Big{)} K(z)dz \\
\geq & - 2 m_1^{p-1} \int_{\cD_2} |\Theta_1 \psi(\bar x, z)|^{p-1} K(z)dz.
\end{split}
\end{equation}
Since
$$
|\Theta_1 \psi(\bar x, z)| \leq C |z|^2 + |\nabla \psi(\bar x)| |z|,
$$
we obtain
$$
I_3 \geq -\kappa  \Big{(} \int_{\cD_2} |z|^{2(p-1)} |z|^{-n - sp}dz + |\nabla \psi(\bar x)|^{p-1}\int_{\cD_2} |z|^{p-1} |z|^{-n - sp}dz\Big{)}.
$$
Again, since $\cD_2\subset B_{\tilde\varrho}\setminus B_\delta$, we get
$$
I_3 \geq -\kappa_1 \Big{(} \int_{\delta}^{\tilde \varrho} r^{2(p-1) - sp - 1} dr + |\nabla \psi(\bar x)|^{p-1} \int_{\delta}^{\tilde \varrho} r^{p-1 - sp - 1}\, dr \Big{)},
$$
where $\kappa_1=\kappa_1(n, p, m, m_1)$ is a constant, from which the result follows.
\end{proof}


\subsection{Proof of Theorem~\ref{Tmain-1} in the subquadratic case.} The proof in this case is quite similar to the superquadratic case.
\begin{proof}[Proof Theorem~\ref{Tmain-1} (subquadratic case)]
Let $\gamma<\gamma_\circ(s)$.  We start by considering the doubling function
 $$\Phi(x, y)= u(x)-u(y)-L\varphi_\gamma(|x-y|)- m_1 \psi(x)\quad x, y\in B_2,$$
 where $\psi(x) = [(|x|^2-1)_+]^m, \varphi_\gamma(t)=|t|^\gamma$. Fix $m\geq 3$ and choose $m_1$
 large enough so that $m_1\psi(x)>2$ for all $|x|\geq \frac{3}{2}$.
 From now on we fix this choice of $m$ and $m_1$. Also, let $L$ be large enough to satisfy $L\varphi_\gamma(1/8)>2$.
 We need to show that for some $L$ we have $\Phi\leq 0$ in $B_2\times B_2$.
 Suppose that, for all large $L$, we have $\Phi(\xbar, \ybar)>0$ and
$$\Phi(\xbar, \ybar)=\max_{\bar{B}_2\times \bar{B}_2}\Phi(x, y).$$
As discussed in Theorem~\ref{T1.7}, it can be easily seen that $\xbar\in B_{\frac{3}{2}}$, $\ybar\in B_{\frac{13}{8}}$,
and $|\abar|<\frac{1}{8}$. Also, setting the notation
\begin{gather*}
\cone=\{z\in B_{\delta_0|\abar|}\; :\; |\langle \abar, z\rangle|\geq (1-\delta_0) |\abar||z|\}, \quad \delta_0=\varepsilon,
\quad \delta=\varepsilon_1|\abar|,\\  \cD_1= B_\delta\cap \cone^c, \quad \cD_2=B_{\frac{1}{8}}\setminus (\cD_1\cup\cone),
\end{gather*}
and $\tilde\varrho=\frac{1}{8}$, we arrive at \eqref{E1.4}. Next we estimate the terms $I_i, i=1,2,3,4$ appropriately to contradict~\eqref{E1.4}.

\medskip

For $I_1$, using Lemma~\ref{lemconesub} together with Lemma~\ref{L1.5}, we have
\begin{equation*}
I_1 \geq c  L^{p-1} |\bar a|^{(\gamma - 1)(p-2) + \gamma - 2} \int_{\cone} |z|^{p - n - sp} \dz,
\end{equation*}
and applying the estimate~\eqref{intconeBarles}, we conclude that
\begin{equation}\label{ET2.9A}
I_1\geq C_{\varepsilon} L^{p-1} |\abar|^{\gamma(p-1)-ps}
\end{equation}
for all $L\geq L_0$, where $C_\varepsilon$ depends only on $n,p,s$ and $\varepsilon$.
For $I_2$, we use Lemma~\ref{lemD1sub} to obtain
\begin{equation*}
I_2 \geq -\kappa L^{p-1} \varepsilon_1^{p(1 - s)}|\bar a|^{\gamma(p-1) - ps }, 
\end{equation*}
and therefore, taking $\varepsilon_1$ small enough in terms of $C_\varepsilon$ in \eqref{ET2.9A}, we arrive at
\begin{equation*}
I_1+I_2\geq \frac{C_{\varepsilon}}{2}L^{p-1} |\abar|^{\gamma(p-1)-ps}
\end{equation*}
for all $L\geq L_0(\dataex)$.

For $I_3$, we use Lemma~\ref{lemD2sub}, more specifically the estimate~\eqref{Rio2} together with the fact that $|\Theta_1 \psi(\bar x, z)| \leq C |z|$, to write 
\begin{align*}
I_3 \geq -\kappa  \int_{\cD_2} |z|^{p-1 - n - sp}dz &\geq -\kappa \int_{\cD_2} |z|^{\gamma(p-1) - n - sp}dz
\\
&\geq
-\frac{\kappa_1}{\gamma(p-1) - sp} \varepsilon_1^{\gamma(p-1) - sp} |\bar a|^{\gamma(p-1)- sp}
\\
&\geq -\frac{\kappa_1}{\gamma(p-1) - sp} \varepsilon_1^{\gamma(p-1) - p} |\bar a|^{\gamma(p-1)- sp},
\end{align*}
by using the definition of $\delta$. Thus, gathering the previous estimates together with the estimate for $I_4$ given in Lemma~\ref{lemI4}, we conclude that
\begin{equation*}
C_{\varepsilon} L^{p-1} |\abar|^{\gamma(p-1)-sp} \leq 2C \varepsilon_1^{\gamma(p-1) - p} |\bar a|^{\gamma(p - 1) - sp} + 2(1+\kappa_4)
\end{equation*}
for all $L\geq L_0$.
Since $\gamma(p-1)-sp<0$ and  $|\bar a|$ tends to zero as $L$ tends to infinity, we arrive at a contradiction for large enough $L$. Notice that the previous estimates can be handled without modifications if we pick $\gamma = \gamma_\circ$ when $\gamma_\circ < 1$, from which the second part of the theorem follows.

\medskip

Now we are left with the case $\gamma_\circ=1<\frac{ps}{p-1}$. The proof is quite similar, but this time we consider the doubling function
$$
\Phi(x, y)= u(x)-u(y)-L \varphi(|x-y|)- m_1 \psi(x)\quad x, y\in B_2,
$$
where $\psi(x) = [(|x|^2-1)_+]^m, \varphi(t)=\tilde\varphi(\frac{r_\circ}{3}t)$ and 
$\tilde\varphi$ satisfies \eqref{logchoice}.
From the first part we know that $u\in C^{0, \upkappa}(B_{\frac{3}{2}})$ for all $\upkappa<1$. For our purpose we fix any $\upkappa\in (0, 1)$.
Now suppose that 
$\sup_{B_2\times B_2}\Phi>0$. For large enough $L$ we would have
$|\xbar|\leq \frac{5}{4}$, $|\ybar|\leq \frac{11}{8}$, $|\abar|<\frac{1}{8}$ and 
$$|\nabla\psi(\xbar)|\leq \kappa |\abar|^{\frac{m-1}{m}\upkappa}.$$
Parameters are also set in a similar fashion:
\begin{gather*}
\phi(x, y)=L\varphi(|x-y|)+m_1\psi(x), \quad \delta_0=\varepsilon (\log^2(|\abar|))^{-1}, \quad \varrho=\frac{\frac{n+1}{2}+p-sp}{p-sp},
\\
\delta=\varepsilon_1 |\abar| (\log^{2\rho}(|\abar|))^{-1}, \quad \cone=\{z\in B_{\delta_0|\abar|} \; :\; |\langle \abar, z\rangle| \geq (1-\delta_0)|\abar||z|\},
\\
\cD_1=B_\delta\cap \cone^c\quad \cD_2= B_{\frac{1}{8}}\setminus(\cD_1\cup\cone).
\end{gather*}
In view of Lemmas~\ref{lemconesub},~\ref{lemD1sub}  and Lemma~\ref{L1.5}-(ii), we can mimic the calculations of the superquadratic case to obtain
$$
I_1+I_2\geq \frac{C_{\varepsilon}}{2} L^{p-1}|\abar|^{p-sp-1}(\log^2(|\abar|))^{-\beta}
$$
for all $L\geq L_0$, for a suitable choice of $\varepsilon$ and $\varepsilon_1$.

For $I_3$, we employ Lemma~\ref{lemD2sub} and the estimate for $\nabla \psi(\bar x)$ to write
\begin{align*}
I_3
&\geq - \kappa_1 \left(\int_\delta^1 r^{2(p-1)-1-ps} \, dr +  C_{\varepsilon_1}|\abar|^{p-1-sp} |\abar|^{\frac{m-1}{m}\upkappa(p-1)}(\log^2(|\abar|))^{\rho(sp+1-p)}\right)
\\
&\geq - \kappa_2 \left(\int_\delta^1 r^{2(p-1)-1-ps} \, dr +  C_{\varepsilon_1}|\abar|^{p-1-sp}(\log^2(|\abar|))^{-\beta} \right)
\end{align*}
for some constants $\kappa_1, \kappa_2, C_{\varepsilon_1}$ and
all $L\geq L_0$ (recall that $|\abar|\to 0$ as $L\to \infty$),
where $C_{\varepsilon_1}$ is a constant depending on $\varepsilon_1$.
The other integration can be computed as follows.
If $2(p-1) -ps >0$, then 
$$\int_\delta^1 r^{2(p-1)-1-ps}\, dr<\kappa_3.$$
If $2(p-1) -ps =0$, then
$$\int_\delta^1 r^{2(p-1)-1-ps}\, dr = -\kappa_3\log|\delta| + C_{\varepsilon_1}
\leq \kappa_4 |\abar|^{p-1-sp}(\log^2(|\abar|))^{-\beta} +
C_{\varepsilon_1}.$$
If $2(p-1) -ps < 0$, then
$$\int_\delta^1 r^{2(p-1)-1-ps}\, dr
\leq C_{\varepsilon_1}|\abar|^{2(p-1)-ps} 
(\log^{2\rho}(|\abar|))^{ps-2(p-1)}
\leq  C_{\varepsilon_1}|\abar|^{p-1-sp}(\log^2(|\abar|))^{-\beta}
$$
for all $L$ large ennough. Combining all these calculations we obtain
\begin{equation*}
I_3\geq - C_{\varepsilon_1}( |\abar|^{p-1-sp}(\log^2(|\abar|))^{-\beta} +1),
\end{equation*}
for $L\geq L_0$.
As before, bound of $|I_4|$ follows from Lemma~\ref{lemI4} . Gathering this estimates in \eqref{E1.4}
we have
$$ \left(\frac{C_{\varepsilon}}{2} L^{p-1} - C_{\varepsilon_1}\right)|\abar|^{p-1-sp}(\log^2(|\abar|))^{-\beta}\leq 2 +C_{\varepsilon_1}+ 2\kappa_4$$
for all $L\geq L_0(\data)$. Since $p-1-sp<0$, the above inequality is not possible for large enough $L$. Hence $\Phi\leq 0$ in $B_2\times B_2$ for all large $L$, proving the theorem.
\end{proof}

\section{Extensions}\label{S-other}
We point out that the proofs above only consider test functions with non-vanishing derivatives at the test points. This gives us the advantage of applying a similar technique to a larger class of operators. For instance, our method can be easily implemented for the fractional $(p,q)$-Laplacian of the form $(-\Delta_p)^{s_1}u + (-\Delta_q)^{s_2}u=f$. 
Note that such operators are intrinsically non-homogeneous with respect to spatial scaling, which makes some classical techniques (e.g., De Giorgi's iteration) difficult to apply. This non-homogeneity makes these operators more challenging to investigate compared to the fractional $p$-Laplacian. Some recent development in this direction can be found in \cite{BOS22,BKO23,DFP19,GKS23}, and references therein.
Since a result analogous to Proposition~\ref{Prop1.3} can be easily proved in this case (more specifically, see \cite{FZ23}), the regularity 
of the viscosity solutions implies the regularity of the weak solutions.
Moreover, since our analysis allows us to treat each nonlocal operator separately, the H\"{o}lder regularity index
for this operator is given by $\varpi=\min\{1, \max\{\frac{s_1p}{p-1}, \frac{s_2 q}{q-1}\}\}$. For precision, we present the following result in the superquadratic case, though other cases (that is, other values of
$p, q, s_1, s_2$) can be treated in a similar fashion. In the context of our next result, we note that \cite{GKS23} studied the interior H\"{o}lder regularity for the the case $2\leq q\leq p<\infty$ and $0<s_2\leq s_1<1$
, establishing that $u\in C^{0, \gamma}_{\rm loc}(B_2)$ for every $\gamma<\gamma_\circ(s_1)=\frac{s_1p}{p-1}$. Our next result shows that $u$ is
actually  almost $\varpi$-H\"{o}lder which is more natural given the structure of the operator.
\begin{thm}\label{T2.10}
Suppose that $p, q\geq 2$ and $u\in C(\bar{B}_2)\cap L^{p-1}_{s_1p}(\Rn)\cap L^{q-1}_{s_2 q}(\Rn)$, $s_1, s_2\in (0, 1)$, is a viscosity solution to 
$$-C\leq (-\Delta_p)^{s_1}u + (-\Delta_q)^{s_2}u\leq C\quad \text{in}\; B_2.$$
Then for any $\gamma< \varpi:=\min\{1, \max\{\frac{s_1p}{p-1}, \frac{s_2 q}{q-1}\}\}$ we have $u\in C^{0, \gamma}(B_1)$. Furthermore, if
$$\sup_{B_2}|u| + \int_{\Rn}\frac{|u(z)|^{p-1}}{(1+|z|)^{n+s_1p}}\dz + \int_{\Rn}\frac{|u(z)|^{q-1}}{(1+|z|)^{n+s_2q}}\dz \leq A_0,$$
then the norm $\norm{u}_{C^{0, \gamma}(B_1)}$ is bounded above by a constant dependent on $C, A_0, n, p, q, s_1, s_2$. In addition, when $\max\{\frac{s_1p}{p-1}, \frac{s_2 q}{q-1}\}\neq 1$, 
$u$ is locally $\varpi$-H\"{o}lder.
\end{thm}

\begin{proof}
Instead of repeating the whole proof, we just mention the key steps involved in the proof. The main idea is to compute the terms $I_i, i=1,2,3,4,$ (see \eqref{E1.4}) separately for $(-\Delta_p)^{s_1}$
and $(-\Delta_q)^{s_2}$. The first step is to note that the self-improving arguments in the first part of Theorem~\ref{Tmain-1} go as it is for $\gamma< \hat{\gamma}:=\min\{1, \min\{\frac{s_1p}{p-1}, \frac{s_2 q}{q-1}\}\}$,
provided at every step in Proposition~\ref{T1.7} we improve from $\upkappa$ to a H\"older exponent $\gamma<\min\{\hat\gamma, \upkappa+\frac{1}{p-1}, \upkappa+\frac{1}{q-1}\}$. This will give
$u\in C^{0, \gamma}_{\rm loc}(B_2)$ for any $\gamma<\hat\gamma$. Now, without loss of generality, 
we may assume that $\frac{s_1p}{p-1}> \frac{s_2 q}{q-1}$ and $\hat\gamma<1$. Note that, when $\frac{s_1p}{p-1}= \frac{s_2 q}{q-1}$,
we have $\hat\gamma= \varpi$, and we are done. Consider $\upkappa<\hat\gamma$ such that
 $\upkappa< \frac{s_2 q}{q-1}< \min\{\upkappa+\frac{1}{q-1}, \upkappa+\frac{1}{p-1}\}$ and $\upkappa(q-2)+1 -qs_2>0$ (this is possible, since $q\geq 2$, $s_2\in (0, 1)$ and $\upkappa$ can be chosen
 arbitrary close to $\frac{s_2q}{q-1}<1$). By our previous argument we have
$u\in C^{0, \upkappa}(B_{2-\nu})$ for some fixed $\nu>0$. Choose $\gamma_1$ such that $\frac{s_2 q}{q-1}<\gamma_1<\min\{\varpi, \upkappa+\frac{1}{p-1}, \upkappa+\frac{1}{q-1}\}$.
With these choice of $\upkappa$ and $\gamma_1$ we can repeat the argument of Proposition~\ref{T1.7}. The estimates of $I^p_i, i=1,..,4,$ (corresponding to $(-\Delta_p)^{s_1}$) will be exactly the same.
Again, fixing $\varepsilon$ and $\varepsilon_1$ suitably \eqref{ET1.7D} holds for $I^q_1+I^q_2$, giving us $I^q_1+I^q_2\geq 0$. Now, by the choice of $\upkappa$ and $\gamma_1$ and \eqref{new}, we see that
$I^q_3\geq -C_{\varepsilon_1}$. Since $I_4$ is bounded above by $\kappa A_0$, for some $\kappa>0$, combining the estimates in \eqref{E1.4}, corresponding to this operator, we would have 
$$
\left[\frac{C_\varepsilon}{2} L^{p-1}-C_{\varepsilon_1}\right] 
|\abar|^{\gamma_1(p-1)-ps}\leq \kappa_1 + C_{\varepsilon_1}
$$
for some constant $\kappa_1$. This clearly cannot hold for large enough $L$, proving that $u\in C^{0, \gamma_1}(B_{2-2\nu})$. Now, we can bootstrap the regularity, following Proposition~\ref{T1.7}, starting from
$\upkappa=\gamma_1$.  This completes the proof of the first part.

$\varpi$-H\"{o}lder regularity follows by applying an argument similar to Theorem~\ref{Tmain-1}.
\end{proof}

Another related model to which our method extends is the nonlocal double phase equation with a 
H\"{o}lder-continuous modulating coefficient. 
This should be compared with those obtained in the local setting. 
Recall that the classical double phase problem corresponds to the minimization of the energy functional
$$\mathcal{F}(u):=\int_\Omega (|\nabla u|^p + \xi(x)|\nabla u|^q)\dx, \quad 1<p\leq q,$$
over a suitable Sobolev space. Such functionals arise in homogenization theory and are also related to the Lavrentiev phenomenon.
For more details, we refer to the recent survey \cite{MR21}. The regularity of the local minimizer of this functional was studied by Colombo and Mingione \cite{CM15},
 where the authors established $C^{1, \beta}$ regularity 
of the minimizer under the condition $1<\frac{q}{p}\leq 1+\frac{\alpha}{p}$, with $\alpha\in (0, 1]$ denoting the H\"{o}lder exponent of $\xi$.
For the corresponding nonlocal model, we establish H\"{o}lder regularity of the solutions under the condition
 $\frac{q}{p}s_2\leq s_1+\frac{\alpha}{p}$, which reduces to the condition of \cite{CM15} when $s_1=s_2=1$. 
 This specific condition for the fractional double-phase problem was also considered in  \cite{BOS22}. The H\"{o}lder exponent
 of the solution in \cite{DFP19, BOS22} is of a qualitative nature and not explicit, whereas we obtain an explicit  {\it optimal} H\"{o}lder
 exponent, considering the possibility of $\xi\equiv 0$. 

\begin{thm}\label{T2.11}
Suppose that $1<p\leq q<\infty$ and $u\in C(\bar{B}_2)\cap L^{p-1}_{s_1p}(\Rn)\cap L^{q-1}_{s_2 q}(\Rn)$ is a viscosity solution to 
$$
 -C\leq (-\Delta_p)^{s_1}u + \xi(x)(-\Delta_q)^{s_2}u\leq C\quad \text{in}\; B_2.
$$

Suppose that $\xi \in C^{0, \alpha}(B_2)$ is nonnegative. Also, let $p\leq q, s_2q< p s_1 +  \alpha$. Then $u\in C^{0, \gamma}(B_1)$
for any $\gamma<\gamma_\circ(s_1)=\min\{1, \frac{s_1p}{p-1}\}$.  Furthermore, $u\in C^{0, \gamma_\circ(s_1)}(B_1)$, provided
$\frac{s_1p}{p-1}\neq 1$.
\end{thm}

\begin{proof}
First, suppose $p, q\geq 2$.
Define $\hat\gamma:=\min\{1, \frac{ps_1}{p-1}, \frac{qs_2}{q-1}\}$.
Let us denote by $\sL_p=(-\Delta_p)^{s_1}$ and $\sL_q=(-\Delta_q)^{s_2}$. $\sL_p[D]$ and $\sL_q[D]$ are defined in the same way as in 
\eqref{AB01}. The proof basically follows from the argument of the previous corresponding results. We only provide a sketch for the first part and the rest 
of the proof can be worked out easily following Theorem~\ref{Tmain-1}. Consider $1\leq\varrho_1<\varrho_2\leq 2$, and suppose that
$u\in C^{0, \upkappa}(B_{\varrho_2})$ for some $\upkappa\in [0, \hat\gamma)$. 
Let 
$$\gamma<\min\left\{\hat\gamma, \upkappa+\frac{1}{p-1}, \upkappa+\frac{1}{q-1}, \upkappa + \frac{\alpha+s_1p-s_2q}{p-1}\right\}$$
 From here we follow the notation of Proposition~\ref{T1.7}.
We choose $m$ large enough so that
\begin{gather*}
(p-2)\upkappa + \frac{m-1}{m} \upkappa + 1>\gamma(p-1), \quad (q-2)\upkappa + \frac{m-1}{m} \upkappa + 1>\gamma(q-1).
\end{gather*}
Consider the doubling function
$$\Phi(x, y)=u(x)-u(y)-L\varphi_\gamma(|x-y|)-m_1\psi(x)\quad x, y\in B_2.$$
The parameters $\xbar, \ybar,\abar, \delta,\cone, \cD_1, \cD_2$ are same as in Proposition~\ref{T1.7}. By $I^p_i$ we denote the 
quantities in \eqref{E1.4} corresponding to the operator $\sL_p$. With these notation \eqref{E1.4} becomes
\begin{align*}
& \sum_{i=1}^4 I^p_i + \xi(\xbar)(\sL_q[\cone]w_1(\bar x) +\sL_q[\cD_1]w_1(\bar x) +\sL_q[\cD_2]w_1(\bar x)+\sL_q[B^c_{\tilde\varrho}]w_1(\bar x))
\\
&\quad -\xi(\ybar)(\sL_q[\cone]w_2(\bar y) +\sL_q[\cD_1]w_2(\bar y) +\sL_q[\cD_2]w_2(\bar y)+\sL_q[B^c_{\tilde\varrho}]w_2(\bar y))\leq 2C
\end{align*}
which we re-write as
\begin{align}\label{T2.11A}
& \sum_{i=1}^4 I^p_i  + \xi(\xbar)\underbrace{(\sL_q[\cone]w_1(\bar x) +\sL_q[\cD_1]w_1(\bar x) )}_{\mathcal{J}_1}
+ \xi(\xbar) \underbrace{(\sL_q[\cD_2]w_1(\bar x)-\sL_q[\cD_2]w_2(\bar y))}_{\mathcal{J}_2}\nonumber
\\
&\quad + \xi(\ybar) \underbrace{(-\sL_q[\cone]w_2(\bar y) -\sL_q[\cD_1]w_2(\bar y))}_{\mathcal{J}_3} + 
\underbrace{(\xi(\xbar)-\xi(\ybar))\sL_q[\cD_2]w_2(\bar y)}_{\mathcal{J}_4}\nonumber
\\
&\qquad + \underbrace{\xi(\xbar)\sL_q[B^c_{\tilde\varrho}]w_1(\bar x)-\xi(\ybar)\sL_q[B^c_{\tilde\varrho}]w_2(\bar y)}_{\mathcal{J}_5}\leq 2C.
\end{align}
The first sum, that is, $\sum_{i=1}^4 I^p_i$ can be computed as in Proposition~\ref{T1.7}, giving us
$$ \sum_{i=1}^4 I^p_i\geq \left(\frac{C_\varepsilon}{2} L^{p-1}-\kappa(\varepsilon_1,[u]_\upkappa)\right) |\abar|^{\gamma(p-1)-ps_1} -\kappa_2,$$
where $[u]_\upkappa=[u]_{C^{0, \upkappa}(B_{\varrho_2})}$.
For $\mathcal{J}_1$, following the calculation 
of $I_1$ and $I_2$ in Proposition~\ref{T1.7} we obtain
$$\mathcal{J}_1\geq \tilde{C}_\varepsilon L^{q-1} |\abar|^{\gamma(q-1)-s_2q}$$
for all $L\geq L_0$ and suitably fixed $\varepsilon, \varepsilon_1$.  Same estimate also holds for $\mathcal{J}_3$, due to the symmetry
of $\varphi_\gamma$. Because of our choice of $\gamma$ and the lower bound estimate of $I_3$ in Proposition~\ref{T1.7}, we get
$$
\mathcal{J}_2\geq -\kappa_1(\varepsilon_1, [u]_\upkappa) (1+ |\abar|^{\gamma(q-1)-s_2q}).
$$
Thus, if we choose $L$ large enough, we obtain
$$\mathcal{J}_1+\mathcal{J}_2+\mathcal{J}_3\geq -\kappa_1(\varepsilon_1, [u]_\upkappa).$$
For $\mathcal{J}_4$, we calculate, using the H\"{o}lder continuity of $\xi$,
$$|\mathcal{J}_4|\leq \kappa |\abar|^\alpha (\varepsilon_1|\abar|)^{\upkappa(q-1)-s_2q}\leq \kappa_2(\varepsilon_1,[u]_k)|\abar|^{\gamma(p-1)-s_1p}$$
where we used
$$\gamma(p-1)-s_1p< (p-1)\upkappa + \alpha-s_2q\leq (q-1)\upkappa + \alpha-s_2q.$$
Since $|\mathcal{J}_5|\leq \kappa_3$ for some constant $\kappa_3$, depending on the  $L^{q-1}_{s_2q}$ norm, putting these estimates in
\eqref{T2.11A} gives us
$$\left(\frac{C_\varepsilon}{2} L^{p-1}-\kappa(\varepsilon_1,[u]_\upkappa)-\kappa_2(\varepsilon_1,[u]_k)\right) |\abar|^{\gamma(p-1)-ps_1}
-\kappa_1(\varepsilon_1, [u]_\upkappa)\leq 2C+\kappa_3.
$$
But since $\gamma<\frac{ps_1}{p-1}$ and $|\abar|<1$, the above cannot hold for large enough $L$. Thus, our usual argument in Proposition~\ref{T1.7},
gives that $C^{0, \gamma}(B_{\varrho_1})$. Applying the bootstrapping argument in Theorem~\ref{Tmain-1} gives $u\in C^{0, \gamma}(B_1)$
for any $\gamma<\hat\gamma$. Now if $\hat\gamma=\gamma_\circ(s_1)$ or $\hat\gamma=1$, then we are done. So we consider
the case $\hat\gamma<1$ and $\frac{s_2q}{q-1}<\frac{s_1p}{p-1}$. This is similar to the situation appearing in Theorem~\ref{T2.10}.
Consider $\upkappa<\hat\gamma$ such that
 $$\upkappa< \frac{s_2 q}{q-1}< \min\left\{\upkappa+\frac{1}{q-1}, \upkappa+\frac{1}{p-1}, \upkappa + \frac{\alpha+s_1p-s_2q}{p-1}\right\}$$
 and $\upkappa(q-2)+1 -qs_2>0$, and let
 $$
\frac{s_2 q}{q-1}< \gamma_1< \min\left\{\gamma_\circ(s_1), \upkappa+\frac{1}{p-1}, \upkappa+\frac{1}{q-1}, \upkappa + \frac{\alpha+s_1p-s_2q}{p-1}
\right\}.
 $$
 Because of this choice of $\upkappa$ and $\gamma_1$, $\mathcal{J}_1, \mathcal{J}_3$ become non-negative and 
 $\mathcal{J}_2\geq -C_{\varepsilon_1}$. $|\mathcal{J}_4|$ enjoys a similar bound as above. Therefore, our previous argument goes through, giving us
 a contradiction for large enough $L$. This would give us $C^{0, \gamma_1}$ regularity of $u$ with $\gamma_1>\frac{s_2 q}{q-1}$. Now we can
 bootstrap in a similar fashion starting with $\upkappa=\gamma_1$ to obtain $C^{0, \gamma}$ regularity for any $\gamma<\gamma_\circ$.

Now suppose $1<p<2$ and $q\geq 2$.
For given $\upkappa\geq 0$, we can choose 
$$\gamma<\min\{\hat\gamma, \upkappa+\frac{1}{q-1}, \upkappa + \frac{\alpha+s_1p-s_2q}{p-1}\},$$
to improve the regularity. After arranging the terms as in \eqref{T2.11A}, it is easily seen that
$I_1+I_2+I_3$ dominates $\mathcal{J}_4$, and therefore, the previous argument goes through. This would help us to improve the 
H\"{o}lder regularity to any exponent strictly smaller than $\hat\gamma$. Next, supposing $\hat\gamma<1$ and $\frac{qs_2}{q-1}<\frac{ps_1}{p-1}$, we can again follow the
argument as above to complete the proof.

The proof for the case $1<p\leq q<2$ and the last part of the proof follow in a similar fashion as in Theorem~\ref{Tmain-1}.
\end{proof}

The proofs of Theorems~\ref{T2.10} and ~\ref{T2.11} can be modified to incorporate the following extensions
\begin{itemize}
\item It is easily seen that Theorem~\ref{T2.11} can further be extended for the model
$$(-\Delta_p)^{s_1}u + {\rm PV}\int_{\Rn} |u(x)-u(x+z)|^{q-2}(u(x)-u(x+z))\frac{\xi(x, z)}{|z|^{n+s_2q}}\dz,$$
where $\xi>0$ is $\alpha$-H\"{o}lder in $x$,  uniformly in $z$, $\xi(x, z)=\xi(x, -z)$, and $p, q, s_1, s_2$ satisfies the conditions mentioned in Theorem~\ref{T2.11}.

\item Given nonnegative $\alpha$-H\"{o}lder continuous functions $\xi_1, \xi_2$ satisfying $\xi_1+\xi_2\geq 1$ we can consider
$$
 -C\leq \xi_1(x) (-\Delta_p)^{s_1}u(x) + \xi_2(x)(-\Delta_q)^{s_2}u(x)\leq C\quad \text{in}\; B_2.
$$
It can be easily seen for the above proof that $u\in C^{0, \hat\gamma}(B_1)$, where 
$\hat\gamma=\min\{1, \min\{\frac{s_1p}{p-1},\frac{s_2q}{q-1}\}\}$, provided that $\min\{\frac{s_1p}{p-1},\frac{s_2q}{q-1}\}\neq 1$.
\end{itemize}


\section{Lipschitz regularity in the critical case $sp = p-1$}\label{S-crit}
Now we deal with the case $\gamma_\circ = \frac{sp}{p-1} = 1$, which we refer to as the critical case. 
As we described in the proofs of Theorem~\ref{Main}, we can set our problem in the viscosity sense and provide the  proof of Theorem~\ref{Main2} through the following
\begin{thm}\label{T-3.1}
Let $p\in (1, \infty)$, $s\in (0,1)$ and $\frac{sp}{p-1}=1$. Assume that $k(z):=K(z)|z|^{n+ps}\in C^{0, \tilde\alpha}(\Rn)$
for some $\tilde\alpha>0$.
Let $u\in C(\bar{B}_2)\cap L^{p-1}_{sp}(\Rn)$ be a viscosity solution to
$$ \sL u=f \quad \text{in}\; B_2,$$where $f$ satisfies \eqref{F} in $B_2$, that is,
\begin{equation*}
W:=\sup_{t\in (0, 1]} \omega_{B_2}(t) |\log(t)|^{n+1+2p(1-s)}<\infty.
\end{equation*}

Then $u\in C^{0, 1}(B_1)$ and
$$\norm{u}_{C^{0,1}(B_1)}\leq C_1 \left(\sup_{B_2}|u| +  \tail(u; 0, 2) + W^{\frac{1}{p-1}} \right)$$
for some constant $C_1$, dependent only on $\lambda, \Lambda, n, p, s, \alpha$, and $\tilde\alpha$-H\"{o}lder
norm of $k$. Furthermore, if 
$$-C\leq  \sL u\leq C\quad \text{in}\; B_2,$$
then there exist $\uptheta>0$, dependent on $n, p,s$, and a constant $C_1$ such that 
$$|u(x)-u(y)|\leq C_1 \left(\sup_{B_2}|u| +  \tail(u; 0, 2) + C^{\frac{1}{p-1}} \right) |x-y|(1+|\log|x-y||^{\uptheta})$$
for $x, y\in B_1$.
\end{thm}

The broad idea of the proof remains same as Theorem~\ref{Tmain-1}. The key difference comes from the careful choice of regularizing function $\varphi$ and the 
calculations associated to it. Let $M=\sup_{B_2} |u| +  \tail(u; 0, 2) + W^{\frac{1}{p-1}} $ .
Replacing $u$ by $u/M$, it is enough to consider the following situation
\begin{equation}\label{ET3.1A}
\begin{split}
	\sL u=f & \quad \text{in}\; B_2,
	\\
	\sup_{B_2}|u| +  \tail(u; 0, 2) &\leq 1,
	\\
	\sup_{t\in (0, 1]} \omega_{B_2}(t) |\log(t)|^{n+1+2p(1-s)}&\leq 1.
	\end{split}
\end{equation}
Let $\tilde\varphi(t)=t+\frac{t}{\log t}$ and 
the regularizing function is defined as $\varphi_c (t)=\tilde\varphi (\frac{r_\circ}{3} t)$, see \eqref{logchoice}. By Theorem~\ref{Tmain-1} we know that $u\in C^{\upkappa}(B_{\frac{7}{4}})$ for all $\upkappa\in (0, 1)$. We fix
	$\upkappa\in (0, 1)$ so that $\upkappa + \upkappa(p-2) + 1 -ps>0$ for $p\geq 2$. This is possible since $p-1=sp$.
	As in Theorem~\ref{Tmain-1}, we define the doubling
	function
	\begin{equation*}
		\Phi(x, y)= u(x)-u(y)-L\varphi_c(|x-y|)- m_1 \psi(x)\quad x, y\in B_2,
	\end{equation*}
	where $\psi(x) = [(|x|^2-\varrho^2_1)_+]^{m}$. Fix $m\geq 3$ large enough so that 
	$$\frac{m-1}{m}\upkappa + \upkappa(p-2) + 1 -ps>0.$$
	Also, set $m_1$ large enough so that $m_1 \psi(x)> 2$ for $|x|\geq 3/2$. We claim that there exists $L$, independent of $u$, such that
	$\Phi\leq 0$ in $B_2\times B_2$. Suppose, to the contrary, that $\sup_{B_2\times B_2}\Phi>0$ for all $L$ large enough. Set $L$ large enough such that $L\varphi_c(\frac{1}{16})> 2$. 
	By our choice of $m_1$, we find $|\xbar|<\frac{3}{2}$, $|\ybar|\leq \frac{3}{2}+\frac{1}{16}$, $|\abar|=|\xbar-\ybar|\leq \frac{1}{16}$ satisfying 
\begin{equation}\label{new-xy}
\sup_{B_2\times B_2}\Phi=\Phi(\xbar,\ybar)>0.
\end{equation}
As before, we denote
\begin{equation*}
\begin{gathered}
\phi(x, y)=L\varphi_c(|x-y|)+m_1\psi(x), \quad \delta_0=\varepsilon (\log^2(|\abar|))^{-1}, \quad \rho=\frac{\frac{n+1}{2}+p-sp}{p-sp}=\frac{n+3}{2},
\\
\delta=\varepsilon_1 |\abar| (\log^{2\varrho}(|\abar|))^{-1}, \quad \cone=\{z\in B_{\delta_0|\abar|} \; :\; |\langle \abar, z\rangle| \geq (1-\delta_0)|\abar||z|\},
\\
\cD_1=B_\delta\cap \cone^c,\quad \cD_2= B_{\frac{1}{8}}\setminus(\cD_1\cup\cone),
     \quad \beta = \frac{n+1}{2}  + p-sp.
\end{gathered}
\end{equation*}
	
This time we will require some regularity of the function $f$, from which inquality \eqref{E1.4} has the specific structure
\begin{align}\label{ET3.1B}
		I_1 + I_2 + I_3 + I_4 \leq 
f(\xbar)-f(\ybar)\leq \omega_{B_2}(|\abar|)\leq \frac{1}{(\log^2(|\abar|))^\beta}
\end{align}
by \eqref{ET3.1A}, where $w_1, w_2$ are defined as before. We estimate $I_i, i=1,2,3,4,$ suitably to reach a contradiction.

We  require further estimates for the contributions of $I_3$ and $I_4$ in the Ishii-Lions method. In view of the previous theorems, 
We fix $\upkappa$ as chosen above.
\begin{lem}[Estimate on $\cD_2$]\label{lemD2critical}
Let $1 < p < \infty$ and $sp=p-1$.
Assume that $(\bar x, \bar y)$ satisfies~\eqref{new-xy} and  $u \in C^{0, \upkappa}(B_{\frac{7}{4}})$ for some $\upkappa\in (0, 1)$ satisfying $\upkappa + \upkappa(p-2) + 1-sp>0$. Then, there exist $C > 0$ and $\theta_1=\theta_1(p, s, \upkappa) \in (0,1)$ such that
\begin{equation*}
I_3 \geq - C |\bar a|^{\theta_1}
\end{equation*}
for all $L\geq L_0=L_0(p, s, \upkappa)$, where $C$ depends on $[u]_\upkappa, \data$ and
$\varepsilon_1$ (see Lemma~\ref{lemD1} and ~\ref{lemD1sub}).
\end{lem}

\begin{proof}
Consider $\theta \in (0,1)$ and denote $\tilde \delta = |\bar a|^\theta$. For $L$ large enough we would have $\delta<\tilde\delta$.
We perform the following extra splitting of $I_3$:
\begin{equation}\label{ET3.1F}
		I_3=\underbrace{\sL[\cD_2\cap B_{\tilde\delta}] w_1(\bar x)-\sL[\cD_2\cap B_{\tilde\delta}] w_2(\bar y)}_{=I_{1,3}}
		+ \underbrace{\sL[\cD_2\cap B^c_{\tilde\delta}] w_1(\bar x)-\sL[\cD_2\cap B^c_{\tilde\delta}] w_2(\bar y)}_{=I_{2,3}}
	\end{equation}
    
First we let $p\geq 2$. In this case the estimate of $I_{1,3}$ is exactly the same as in Lemmas~\ref{lemD2},~\ref{lemD2sub}, but replacing $\tilde \varrho$ by $\tilde \delta$. 
More precisely, we write
	\begin{equation*}
	I_{1,3} \geq -C [u]_\upkappa^{p-2} \Big{(} \int_{\delta}^{\tilde \delta} r^{\upkappa(p - 2) - sp + 1} dr + |\nabla \psi(\bar x)| \int_{\delta}^{\tilde \delta} r^{\upkappa(p-2) -sp} dr \Big{)}.
	\end{equation*}
By our assumption on $\upkappa$ we have $\upkappa(p-2) -sp + 2 > 0$, leading to
	\begin{equation*}
	I_{1,3} \geq -C [u]_\upkappa^{p-2} \Big{(} |\bar a|^{\theta(\upkappa(p - 2) - sp + 2)} + |\nabla \psi(\bar x)| \int_{\delta}^{\tilde \delta} r^{\upkappa(p-2) -sp} dr \Big{)}.
	\end{equation*}
On the other hand, if $p > 2$, we  have $\upkappa(p - 2) - sp < - 1$; and, if $p=2$, then $\upkappa(p - 2) - sp = -1$. Thus, we have
$$
\int_{\delta}^{\tilde \delta} r^{\upkappa(p-2) - sp}dr \leq \left \{ \begin{array}{cl} 
  (\varepsilon_1\log^{-2\rho}(|\abar|)^{\upkappa(p-2)-sp+1}|\bar a|^{\upkappa(p-2) - sp + 1} \quad & \mbox{if} \ p > 2, 
 \\ 
\log(\varepsilon^{-1}_1|\bar a|^{\theta-1}\log^{2\rho}(|\abar| ))\quad & \mbox{if} \ p = 2 . \end{array} \right .
$$
Since $|\nabla\psi(\xbar)|\leq \kappa |\abar|^{\frac{m-1}{m}\upkappa}$, there exists $\theta_\upkappa > 0$ such that
$$
I_{1,3} \geq - C (\varepsilon_1, [u]_\upkappa, m, m_1) |\bar a|^{\theta_\upkappa}.
$$
 for all $L\geq L_0$.
Notice that by a slight modification of the previous argument allows us to take $\theta_\upkappa$ independent of $p \geq 2$.

	To calculate $I_{2,3}$, we go one step back in the proof of Lemma~\ref{lemD2}. More specifically, using that
	\begin{align*}
	|\Theta_1 u(\bar x, z)|, |\Theta_1 u(\bar x,z)| & \leq [u]_\upkappa |z|^{\upkappa}, \\
	|\Theta_1 u(\bar x, z) - \Theta_1 u(\bar y, z)| & \leq 2 [u]_\upkappa |\bar a|^\upkappa,
	\end{align*}
	and replacing them into~\eqref{I3superback}, we get
%
%
	\begin{align*}
		I_{2,3} &\geq -\kappa |\abar|^{\upkappa}\int_{B_{\frac{1}{8}}\setminus B_{\tilde\delta}} |z|^{\upkappa(p-2) -n-ps}\dz
		\\
		&\geq -\kappa |\abar|^{\upkappa} (\tilde\delta)^{\upkappa(p-2)-ps} =\kappa |\abar|^{\upkappa+\theta(\upkappa(p-2)-ps)}.
	\end{align*}
Taking $\theta$ small enough so that 
    $\upkappa+\theta(\upkappa(p-2)-ps)>0$, and then  gathering 
    the estimates of $I_{1,3}, I_{1,2}$, we conclude the result for the superquadratic case. 
	
	\medskip

Now consider $p\in(1,2)$. We use the same splitting in~\eqref{ET3.1F}, with $\tilde \delta = |\bar a|^\theta$ for some $\theta$ small. 
For $I_{1,3}$, using the same arguments as in Lemma~\ref{lemD2sub}, we have
	\begin{align*}
		I_{1,3} &\geq - \kappa \Big{(} \int_{\delta}^{\tilde \delta} r^{2(p-1) - sp - 1}dr + |\nabla \psi(\bar x)|^{p - 1} \int_{\delta}^{\tilde \delta} r^{-1} dr \Big{)}
        \\
        &\geq - \kappa (|\bar a|^{\theta(p-1)} + |\bar a|^{(p-1) \upkappa \frac{m-1}{m}} \log (\varepsilon^{-1}_1|\bar a|^{\theta-1}\log^{2\rho}(|\abar|)),
	\end{align*}
	implying, the existence of a $\theta_2 \in (0,1)$  such that
	$$
	I_{1,3} \geq -C_{\varepsilon_1} |\bar a|^{\theta_2}.
	$$
For $I_{2,3}$, we need a little extra care. Using
	\begin{align*}
	I_{2,3} \geq & -\int_{\cD_2 \cap B_{\tilde \delta}^c} |J_p(\Theta_1 u(\bar x, z)) - J_p(\Theta_1 u(\bar y, z))|K(z)dz,
	\end{align*}
	together with the fact that for $z \in \cD_2$ we have
%
$$ |J_p(\Theta_1 u(\bar x, z)) - J_p(\Theta_1 u(\bar y, z))|\leq 2 |u(\ybar) - u(\bar x) - (u(\ybar+z)-u(\xbar+z))|^{p-1}\leq \kappa |\abar|^{\upkappa(p-1)},$$
we arrive at
\begin{align*}
I_{2,3} &\geq -\kappa |\abar|^{\upkappa(p-1)}\int_{B_{\frac{1}{8}}\setminus B_{\tilde\delta}} |z|^{-n-ps}\dz
\\
&\geq -\kappa |\abar|^{\upkappa(p-1)} (\tilde\delta)^{-ps} =\kappa |\abar|^{\upkappa(p-1)-\theta ps}.
\end{align*}
Thus, by taking $\theta$ small
so that $\upkappa(p-1)-\theta sp>0$, and redefining $\theta_2$ smaller, if required, we obtain
$$
I_{2,3} \geq -\kappa |\bar a|^{\theta_2}.
$$
Combining the above estimates, we complete the proof.
\end{proof}

We now estimate $I_4$.
\begin{lem}\label{lemI4critical}
Let $1 < p < \infty$. Assume that $(\bar x, \bar y)$ satisfies~\eqref{new-xy}, and  $u \in C^{0, \upkappa}(B_{\frac{7}{4}})$ for some $\upkappa > 0$. 
Then, there exist $C > 0$ and $\theta_3=\theta_3(\upkappa, \tilde\alpha) \in (0,1)$ such that
$$
|I_4| \leq C|\bar a|^{\theta_3}.
$$
\end{lem}

\begin{proof}
Estimation of $I_4$ becomes more delicate than the one done in Proposition~\ref{T1.7}. We write
	\begin{align}\label{ET3.1H}
		|I_4| &\leq \underbrace{\left|\int_{|\xbar-z|\geq \frac{1}{8}} J_p(u(\xbar)-u(z)) |K(\xbar-z)-K(\ybar-z)| \dz\right|}_{=\cJ_1}\nonumber
		\\
		&\quad + \underbrace{\left|\int_{|\xbar-z|\geq \frac{1}{8}} J_p(u(\xbar)-u(z)) K(\ybar-z) \dz- \int_{|\ybar-z|\geq \frac{1}{8}} J_p(u(\xbar)-u(z)) K(\ybar-z) \dz\right|}_{=\cJ_2}
		\nonumber
		\\
		&\qquad + \underbrace{\left|\int_{|\ybar-z|\geq \frac{1}{8}} (J_p(u(\xbar)-u(z)) - J_p(u(\ybar)-u(z)))K(\ybar-z) \dz\right|}_{=\cJ_3}.
	\end{align}
	Since $|\xbar-\ybar|=|\abar|\leq \frac{1}{16}$, we get $|\xbar-z|\geq \frac{1}{8}\Rightarrow |\ybar-z|\geq \frac{1}{16}$. Therefore, using
	the H\"{o}lder continuity of $k$, we get for $|x-z|\geq\frac{1}{8}$ that
	\begin{align*}
		|K(\xbar-z)-K(\ybar-z)| &\leq |z-\xbar|^{-n-ps} |k(\xbar-z)-k(\ybar-z)|+ k(\ybar-z) ||z-\xbar|^{-n-ps}-|z-\ybar|^{-n-ps}|
		\\
		&\leq \kappa \left(|z-\xbar|^{-n-ps} |\abar|^{\tilde\alpha} + |z-\xbar|^{-n-1-ps}|\abar|\right)
		\\
		&\leq \kappa |\abar|^{\tilde\alpha} |z-\xbar|^{-n-ps}.
	\end{align*}
	By \eqref{ET3.1A} we then obtain $\cJ_1\leq \kappa_1 |\abar|^{\tilde\alpha}$. To estimate $\cJ_2$, we first note that
	we only need to calculate the integration on the set  
	\begin{align*}
		&(\{|\xbar-z|\geq \frac{1}{8}\}\setminus\{|\ybar-z|\geq \frac{1}{8}\})\cup (\{|\ybar-z|\geq \frac{1}{8}\}\setminus\{|\xbar-z|\geq \frac{1}{8}\})
		\\
		&\quad = (\{|\xbar-z|\leq \frac{1}{8}\}\setminus\{|\ybar-z|\leq \frac{1}{8}\})\cup (\{|\ybar-z|\leq \frac{1}{8}\}\setminus\{|\xbar-z|\leq \frac{1}{8}\}).
	\end{align*}
	Since $\sup_{B_2}|u|\leq 1$, it follows that $\cJ_2\leq \kappa_3 |\abar|$. To compute $\cJ_3$, we see that, for $p\geq 2$, 
	$$
	|J_p(u(\xbar)-u(z)) - J_p(u(\ybar)-u(z))|
	\leq \kappa (|u(z)|+1)^{p-2} |u(\xbar)-u(\ybar)|\leq \kappa_4 (|u(z)|+1)^{p-2} |\abar|^{\upkappa}.
	$$
	Hence, by \eqref{ET3.1A},
	\begin{align*}
		\cJ_3 &\leq \kappa_4 |\abar|^{\upkappa}\int_{|\ybar-z|\geq \frac{1}{8}}\frac{(|u(z)|+1)^{p-2}}{|z-\ybar|^{n+ps}}\dz
		\\
		&\leq \kappa_4 |\abar|^{\upkappa} \left[ \int_{|\ybar-z|\geq \frac{1}{8}} \frac{(|u(z)|+1)^{p-1}}{|z-\ybar|^{n+ps}}\dz \right]^{\frac{p-2}{p-1}}
		\left[ \int_{|\ybar-z|\geq \frac{1}{8}} \frac{1}{|z-\ybar|^{n+ps}}\dz \right]^{\frac{1}{p-1}}
		\\
		&\leq \kappa |\abar|^{\upkappa}
	\end{align*}
	for some constants $\kappa_4, \kappa$. For $p\in (1,2)$, we use 
$$
|J_p(u(\xbar)-u(z)) - J_p(u(\ybar)-u(z))|
\leq 2 |u(\xbar)-u(\ybar)|^{p-1} \leq \kappa |\abar|^{\upkappa(p-1)},
$$
giving us $\cJ_3\leq \kappa |\abar|^{\upkappa(p-1)}$. 
Plug-in these estimates in \eqref{ET3.1H} we have
	\begin{equation*}
		|I_4|\leq \kappa \max\{|\abar|^{\tilde\alpha}, |\abar|^{\upkappa}, |\abar|^{\upkappa(p-1)}\},
	\end{equation*}
which completes the proof.
\end{proof}

Now we are ready to provide the
\begin{proof}[Proof of Theorem~\ref{T-3.1}]	
For $I_1$, using Lemma~\ref{lemcone} and Lemma~\ref{L1.5}-(ii), we have
$$
I_1 \geq C_\varepsilon L^{p-1} (\log^2 |\bar a|)^{-\beta}.
$$
For $I_2$, we use Lemma~\ref{lemD1}, similar to the proof of Theorem~\ref{Tmain-1},
	$$
	I_2 \geq -c L^{p-1} \varepsilon_1^{p(1 - s)} (\log^2(|\abar|))^{-\beta}.
	$$
For $I_3$, with the choice of $\upkappa$ in Lemma~\ref{lemD2critical} we obtain
	\begin{equation*}
		I_3\geq -C |\bar a|^{\theta_1}.
	\end{equation*}
For $I_4$, we use Lemma~\ref{lemI4critical} to get
$$
I_4 \geq -C |\bar a|^{\theta_3}.
$$
Now we first choose $\varepsilon_1$ small enough in terms of
$C_\varepsilon$, then combine the estimates in~\eqref{ET3.1B}
arrive at
\begin{align*}
\frac{1}{2} C_\varepsilon L^{p-1} (\log^2(|\abar|))^{-\beta} \leq C_{\varepsilon_1} \max\{|\bar a|^{\theta_1},|\abar|^{\theta_3}, (\log^2(|\abar|))^{-\beta}\},
	\end{align*}
for all $L\geq L_0$, where $L_0$ is a suitable constant dependent on $\dataex$.
Since $|\abar|\to 0$ as $L\to\infty$, the above inequality leads to contradiction and therefore, the proof follows. This completes the proof of the first part.

\medskip

We finish with the logarithmic correction of Lipschitz bounds when $f$ is merely bounded. Dividing $u$ by $M=\sup_{B_2} |u| +  \tail(u; 0, 2) + C^{\frac{1}{p-1}} $, we assume that $C=1$ and \eqref{E1.1} holds.

We only provide a proof for $p\geq 2$, and proof for $p\in (1,2)$ will be analogous.
As before, we consider the regularizing function $\varphi_c$, the localization function $\psi$,
and the doubling function $\Phi$ from the first part. Let
$$
\mathscr{M} := \mathscr{M}(L)= \sup_{B_2 \times B_2} \Phi.
$$

If $\mathscr{M} \leq 0$ for all $L$ large enough, we get that $u$ is Lipschitz continuous in $B_1$ and our required estimate holds.
Thus, we assume that $\mathscr{M} > 0$ for some $L$ large enough. 
From the calculations in the first part and
gathering the estimates of $I_1, ..., I_4$ into~\eqref{ET3.1B},  we arrive at
\begin{align*}
\frac{1}{2} C_\varepsilon L^{p - 1} (\log^2 |\bar a|)^{-\beta} \leq & 2  + \kappa \max\{|\abar|^{\theta_1}, |\abar|^{\theta_3}\}  .
\end{align*}
The choice of $\upkappa$ is fixed as in the first part
(that is, Lemma~\ref{lemD2critical} and ~\ref{lemI4critical}) so that the exponents of $|\abar|$ in the above expression are all positive.
Since $|\abar|\to 0$ as $L\to \infty$, we can choose $L_0>1$ large enough so that 
$$ \frac{1}{4} C_\varepsilon L^{p - 1} |\log |\bar a||^{-2\beta}\leq 2\Rightarrow  |\bar a|\leq e^{-\vartheta_1 L^{\vartheta_2}}$$
for all $L\geq L_0=L_0(\data)$, where $\vartheta_2=\frac{p-1}{2\beta}$ and $\vartheta_1=(\frac{C_\varepsilon}{8})^{\frac{1}{2\beta}}$.
Coming back to the definition of $\mathscr{M}$, we see that
if $\mathscr{M}(L)>0$ for some $L\geq L_0$, then for all $x, y \in B_1$ we have
$$
u(x) - u(y) - L \varphi(x - y) - \psi(x) \leq u(\bar x) - u(\bar y) - L \varphi(\bar x - \bar y) - \psi(\bar x),
$$
implying,
\begin{equation}\label{AB007}
u(x) - u(y) \leq 2L |x - y| + C_\upkappa |\bar x - \bar y|^\upkappa \leq L |x - y| + C_\upkappa e^{-\vartheta_1 \upkappa L^{\vartheta_2}}.
\end{equation}

It is also important to note that \eqref{AB007} holds, even if $\mathscr{M}(L)\leq 0$. Thus, \eqref{AB007} holds
for all $L\geq L_0$ and for all $x, y\in B_1$.
Note that the map $L\mapsto L |x - y| + C_\upkappa e^{-\vartheta_1 \upkappa L^{\vartheta_2}}$ attains minimum at the point $\tilde{L}$ satisfying
$$|x-y|= C_\upkappa\,\vartheta_2\, e^{-\vartheta_1 \upkappa \tilde{L}^{\vartheta_2}} \tilde{L}^{\vartheta_2-1}.$$
It is easily seen that there exists $\kappa(L_0)\in (0, 1)$ such that for any $|x-y|\leq \kappa(L_0)$, the largest $\tilde{L}$ satisfying the above 
equations must be larger than $L_0$. Furthermore, 
$$|x-y|= C_\upkappa\,\vartheta_2\, e^{-\vartheta_1 \upkappa \tilde{L}^{\vartheta_2}} \tilde{L}^{\vartheta_2-1}
\leq e^{-\frac{\vartheta_1}{2} \upkappa \tilde{L}^{\vartheta_2}} \tilde{C}_\upkappa\Rightarrow
\tilde{L}\leq \left[-\frac{2}{\upkappa\vartheta_1}\log\left(\frac{|x-y|}{\tilde{C}_\upkappa}\right)\right]^{\frac{1}{\vartheta_2}},
$$
where $\tilde{C}_\upkappa=C_\upkappa \max_{L\geq 1} e^{-\frac{\vartheta_1}{2} \upkappa \tilde{L}^{\vartheta_2}}\tilde{L}^{\vartheta_2-1}$.
Thus, for $|x-y|\leq \kappa(L_0)$ we obtain
\begin{align*}
|u(x)-u(y)|\leq \tilde{L}|x-y| + \frac{1}{\vartheta_2}|x-y|\tilde{L}^{1-\vartheta_2}\leq \kappa |x-y|
\left[-\log\left(\frac{|x-y|}{\tilde{C}_\upkappa}\right)\right]^{\frac{1}{\vartheta_2}}
\end{align*}
for some constant $\kappa$. Since $\sup_{B_2}|u|\leq 1$, we can find $C_1$ such that
$$|u(x)-u(y)|\leq C_1 |x-y|\,(1+|\log|x-y||^{\frac{1}{\vartheta_2}})\quad \text{for all}\; x, y\in B_1.$$
This completes the proof.
\end{proof}

\begin{proof}[Proof of Theorem~\ref{Main2}]
Since $u\in C(\Omega)$ by the arguments of Theorem~\ref{Main},
the proof follows from Proposition~\ref{Prop1.3} and Theorem~\ref{T-3.1}.
\end{proof}

\bigskip

\subsection*{Acknowledgement}
We are grateful to the referees for their constructive comments and suggestions, which have improved the quality and presentation of the article.
We also thank Fernando Quiros for bringing this question of regularity to our notice, and Olivier Ley for pointing out the proof of the log-Lipschitz estimates in Theorem~\ref{T-3.1}.
This research of Anup Biswas was supported in part by a SwarnaJayanti
fellowship SB/SJF/2020-21/03. Erwin Topp was supported by CNPq Grant 306022. Both authors were also supported by a CNPq Grant 408169. 

\end{document}